\documentclass[11pt]{article}
\usepackage{amsmath,amssymb,amscd,amsfonts,amsthm,mathrsfs,eucal,dsfont,verbatim}
\usepackage{textcomp}
\usepackage{paralist}
\usepackage[english]{babel}

\usepackage{color}%for comments

\newcommand{\ax}{\mathbb{R}^+}

\newcommand {\ov}[1]{ \overline{#1}}
\newcommand{\rd}{{\mathbb{R}^d}}
\newcommand{\rone}{\mathbb{R}}
\renewcommand{\Re}{\rone}
\renewcommand{\star}{\circledast}

\newcommand{\E}{\mathbf{E}}
\newcommand{\wt}{\widetilde}

\newcommand{\prt}{\partial}
\renewcommand {\epsilon}{\varepsilon}

\newcommand{\1}{\mathds{1}}

\newcommand{\eps}{\varepsilon}

\newcommand{\sgn}{\mathrm{sgn}}

\newcommand{\Lba}{\widetilde{L}}

\theoremstyle{plain}
\newtheorem{thm}{Theorem}[section]
\newtheorem{prop}{Proposition}[section]
\newtheorem{cor}{Corollary}[section]
\newtheorem{lem}{Lemma}[section]
\theoremstyle{definition}
\newtheorem{dfn}{Definition}[section]
\newtheorem{ex}{Example}[section]
\theoremstyle{remark}
\newtheorem{rem}{Remark}[section]

\usepackage{hyperref}

\DeclareMathSymbol{\ophi}{\mathalpha}{letters}{"1E}

\renewcommand{\phi}{\varphi}

\newcommand{\be}{\begin{equation}}
\newcommand{\ee}{\end{equation}}
\newcommand{\ben}{\begin{equation*}}
\newcommand{\een}{\end{equation*}}

\newcommand{\ba}{\begin{aligned}}
\newcommand{\ea}{\end{aligned}}

\newfont{\cyrfnt}{wncyr10}
\def\J3{\cyrfnt{\rm \u{\cyrfnt I}}}
\def\j3{\cyrfnt{\rm \u{\cyrfnt i}}}

\usepackage[]{color}
\definecolor{DarkGreen}{rgb}{0.1,0.7,0.3}   %define a custom color

\definecolor{DarkGreen}{rgb}{0.1,0.7,0.3}   %define a custom color

\numberwithin{equation}{section}

\allowdisplaybreaks[4]

\begin{document}

\title{Approximation in law of locally $\alpha$-stable L\'evy-type processes by non-linear regressions}
\author{Alexei Kulik\footnote{Wroclaw University of Science and Technology, Faculty of Pure and Applied Mathematics,
Wybrze\'ze Wyspia\'nskiego Str. 27, 50-370 Wroclaw, Poland, and} \footnote{Institute of  Mathematics, Ukrainian National Academy of Sciences,
Tereshchenkivska Str.\ 3, 01601 Kiev, Ukraine; kulik.alex.m@gmail.com}\
\thanks{The research is supported by Wrocław University
of Science and Technology, grant 0401/0155/18}
}

\maketitle
\begin{abstract}
We study a real-valued L\'evy-type process $X$, which is \emph{locally $\alpha$-stable} in the sense that its jump kernel is a combination of  a `principal' (state dependent) $\alpha$-stable part with a `residual' lower order  part. We show that under mild conditions on the  local characteristics of a process (the jump kernel and the velocity field) the process is uniquely defined, is Markov, and has the strong Feller property.  We approximate $X$  in law  by a non-linear regression $\wt X^x_{t}=\mathfrak{f}_t(x)+t^{1/\alpha}U^{x}_t
$
with a deterministic \emph{regressor term} $\mathfrak{f}_t(x)$ and $\alpha$-stable \emph{innovation term} $U^{x}_t$, and provide error estimates for such an approximation. A case study is performed, revealing different types of assumptions which lead to various choices of regressor/innovation terms and various types of the estimates.  The assumptions are quite general, cover the  \emph{super-critical case $\alpha<1$}, and allow non-symmetry of the L\'evy kernel and unboundedness of the drift coefficient.
\end{abstract}
\section{Introduction}

  L\'evy processes are used nowadays in a wide  variety of models in  physics, biology,  finance etc., where the random noise - by different reasons - can not be assumed Gaussian, and thus the entire model does not fit to the diffusion framework. For instance, the famous Ditlevsen model of the millennial climate changes \cite{Dit99} is based on the observation  that the available ice-core data necessarily requires non-Gaussian noise to be  included into the model. In the basic Ditlevsen model the non-Gaussian noise is $\alpha$-stable; nowadays it is understood that it would be physically more realistic to have the parameters of the noise state-dependent; e.g. the skewness parameter should be positive in the cold glacial periods and negative in the warmer interstadials. The similar problem appears in many other models  with state-dependent parameters, which gives a natural background for the notion of a \emph{L\'evy-type process}. The latter is understood as a (kind of) a L\'evy process whose characteristic triplet is allowed to depend on the current value  of the process; we refer to \cite{BSW} for a detailed introduction, see also  Section \ref{s2} below.
The definition of a L\'evy-type process has the same spirit with  the classical Kolmogorov's definition of  a diffusion process as a location-dependent Brownian motion with a drift. However, to the contrast with the classical theory of diffusions, in the general theory of such L\'evy-type process some principal questions remain unsolved in general, e.g.
\begin{itemize}
  \item[(I)]   for a given set of local  characteristics, is the corresponding L\'evy-type process uniquely defined?
  \item[(II)] what kind of local properties of the law of the process can be derived, and under which assumptions on characteristics?
\end{itemize}
Not being able to discuss in details a considerable list of references devoted to these questions, we refer to
 \cite{KK15}, \cite{Kuh17-2}, \cite{KKK} for such a discussion, and only note that the available methods contain a considerable list of limitations, which exclude from the consideration many natural and physically relevant L\'evy-type models.

In this paper we provide a detailed study of one class of L\'evy-type processes, which is highly relevant for applications and, on the other hand, reveals numerous hidden challenges which one encounters while trying to resolve the above questions (I), (II) in general L\'evy-type  setting. The class to be studied can be shortly described as a mixture of a real-valued $\alpha$-stable-type process with   state dependent \emph{drift, intensity,}  and \emph{skewness} parameters  on one hand, and  a certain (state dependent) lower order `nuisance' part on the other hand; see a detailed definition in Section \ref{s2}. The $\alpha$-stable noise, because of its scaling property, has an exceptional importance in physical applications, and at the same time there are strong reasons to require  the parameters of the noise to be state-dependent, likewise to the  Ditlevsen model discussed above.
 Presence of the `residual' lower order  part is quite reasonable, as well. Namely, this part allows one to introduce a wide spectrum of tempering/damping effects for the tails of the noise, which combines both the $\alpha$-stable and Gaussian regimes (see \cite{R07}) and thus appear frequently in physical models (see \cite{SCK} and references therein).  On the other hand, a lower order  \emph{microstructural} noise terms   without a specified inner structure appear quite naturally in finance models; see  \cite{AJ07} for a detailed discussion.

 For such a \emph{locally $\alpha$-stable} L\'evy-type model we prove the corresponding process to be uniquely defined and to be a Markov process with strong Feller property, thus resolving the general question (I). To approach the question (II),  we specify a family of $\alpha$-stable probability densities  $g^{t,x}, t> 0, x\in \Re$ and a function $\mathfrak{f}_t(x), t\geq 0,x\in \Re$ such that the  transition density $p_t(x,y)$ of $X$ has representation
\be\label{p_rep}
p_t(x,y)={1\over t^{1/\alpha}}g^{t,x}\left({y-\mathfrak{f}_t(x)\over t^{1/\alpha}}\right)+R_t(x,y),
\ee
where the residual kernel $R_t(x,y)$ is negligible (in a certain sense) as $t\to 0$. This representation essentially means that, conditioned by $X_0=x$, process $X$ admits approximation  in law by the non-linear regression
\be\label{regression}
 \wt X^x_{t}=\mathfrak{f}_t(x)+t^{1/\alpha}U^{x}_t, \quad t>0,
\ee
where $U^{t,x}$ is a random variable with the $\alpha$-stable distribution density $g^{t,x}$. We call $\mathfrak{f}_t(x)$ a (deterministic) \emph{regressor term} for $X$, and  $U^{t,x}$  an \emph{$\alpha$-stable innovation term}. It is natural to call \eqref{regression} a \emph{conditionally $\alpha$-stable approximation} to $X$, in the same spirit with the standard \emph{conditionally Gaussian approximation} for a diffusion. However, we will see that  the regressor term $\mathfrak{f}_t(x)$ in general should have  a more sophisticated form than just $x+b(x)t$, typical for the diffusion case.

Our study is based on  the  \emph{parametrix method}, which in the diffusion case is a classical analytical tool to construct and investigate  transition densities. To apply this method in the (non-Gaussian) L\'evy-type setting, we modify it  substantially; here we outline the most crucial change. The classical parametrix method  relies on the fact that a (properly chosen) `zero order approximation' to the unknown $p_t(x,y)$ and corresponding `differential error term' (see Section \ref{s41} below for these definitions) follow certain  prior bounds, which then propagate to the transition density $p_t(x,y)$. For diffusions these kernels are Gaussian; for certain  $\alpha$-stable-type models  similar kernel estimates with  $\alpha$-stable kernels are available as well;  see \cite{Ko89} and \cite{Ko00} for the cases $\alpha>1$ and  $\alpha\leq 1, b\equiv 0$ respectively, and \cite{Ku18}, where in the technically more involved  \emph{super-critical regime} with  $\alpha<1$ and non-trivial $b$ the kernel estimates are obtained as a combination of stable kernels with deterministic flows. However, all these models are `regular' in the  sense that the L\'evy kernel of the noise is assumed to have a density w.r.t. the Lebesgue measure. Presence of singular terms may change situation drastically; see  Example \ref{ex1} below, where $p_t(x,\cdot)$ is  unbounded and thus kernel estimates simply fail. The same effect have been discussed in the recent preprint \cite{KulStoRyz18} for solutions of multidimensional SDEs with cylindrical $\alpha$-stable noise and non-trivial rotation, see \cite[Remark~4.23]{KulStoRyz18}.

To study such highly singular settings, we adopt the following two-stage scheme. First, we
establish integral-in-$y$ estimates (actually, operator norm estimates in $C_\infty$) and perform the parametrix method with the convergence of corresponding series understood in this ($L_1$) sense. This resolves question (I) and gives $L_1$-estimates for the error term $R_t(x,y)$ in \eqref{p_rep}. Second, we analyse the series representation for $p_t(x,y)$ and clarify additional assumptions, which one should require in order  to get stronger types of estimates for $R_t(x,y)$: uniform-in-$(x,y)$ and kernel estimates. This scheme is motivated by perspective applications, where the choice among several types of estimates will allow one to avoid limitations in the model’s assumptions when a particular
application is considered. We plan to use integral-in-$y$ estimates in the proof of Local Asymptotic (Mixed) Normality property for statistical models with discretely observed L\'evy-type processes (this is an ongoing project with A.Kohatsu-Higa) and, combined with uniform-in-$(x,y)$ estimates, in the asymptotic study of the Least Absolute Deviation  estimator for a drift parameter (this is an ongoing project with H.Masuda). Motivated by these applications, we restrict the current exposition by the real-valued case with constant $\alpha$.   The multidimensional locally stable-like model with state-dependent $\alpha=\alpha(x)$ is considered in the widest  generality in the  companion paper \cite{KKS18}.

 The structure of the paper is the following. In Section \ref{s2} we introduce the notation and specify the model. In Section \ref{s3} we specify the conditions and formulate the main results. For these results we also provide a discussion, including examples, possible extensions, and related references. In Section \ref{s4} we separately explain the essence of the  parametrix method and derive the corresponding integral representation of the (candidate for) the transition probability density of the required process. Sections  \ref{s5} -- \ref{s7} respectively contain the proofs of three main results, Theorem \ref{mainthm1} -- Theorem \ref{mainthm3}. The proofs of certain technicalities, which otherwise would make the reading much more difficult, are postponed to Appendix.

\noindent \textbf{Acknowledgments}. The author is grateful to Arturo Kohatsu-Higa and Hiroki Masuda for insightful discussions which clarified for  him the perspectives of the parametrix technique in application to statistics, and for numerous helpful comments about the previous version of the draft. The author gladly expresses a particular respect to Victoria Knopova and Ren\'e Schilling: the numerous discussion on this paper and on the companion one \cite{KKS18} made a deep impact on the style of the final exposition and  saved the author from several pitiful mistakes. The work on the project has been principally finished  during the visit of the author  to the Technical University of Dresden (Germany); the author is  very grateful  to the Technical University of Dresden and especially to Ren\'e Schilling for their support and hospitality.    Finally, the author is deeply grateful to referees for their attention to the paper and very helpful comments.

\section{Notation and preliminaries}\label{s2}  In what follows, $C_\infty$ denotes the class of continuous functions $\Re\to \Re$ vanishing at $\infty$, and $C_0$ denotes the class of continuous functions with compact support. By $C^2_\infty, C^2_0$ we denote the classes  of twice differentiable functions $f$ such that $f, f', f''$ belong to  $C_\infty$ or $C_0$, respectively. A \emph{L\'evy-type operator} $L$ with the domain $C_\infty^2$ is defined by
\be\label{L_gen}
Lf(x)=b(x)f'(x)+{1\over 2}a(x)f''(x)+\int_\Re\Big(f(x+u)-f(x)-uf'(x)1_{|u|\leq 1}\Big)\mu(x; du),\quad f\in C_\infty^2.
\ee
Here $b:\Re\to \Re,a:\Re\to \Re^+$ are given measurable functions,  and $\mu(x; du)$ is a \emph{L\'evy kernel}; that is, a measurable function w.r.t.  $x$ and a L\'evy measure w.r.t. $du$.

There are two natural and closely related ways to associate a {L\'evy-type process} $X$ with the L\'evy type operator $L$. Within the  first one,  $X$ is  a time-homogeneous Markov process which generates a Feller semigroup (that is, a strongly continuous semigroup in $C_\infty$) such that its generator $A$ coincides with $L$ on $C_\infty^2$ (or, which is slightly more general, on $C_0^2$). The second way is based on the notion of the Martingale Problem (MP). Recall that a process $X$ is said to be a \emph{solution} to the \emph{martingale problem} $(L,\mathcal{D})$, if for every  $f\in \mathcal{D}$
the process
$$
f(X_t)-\int_0^tLf(X_s)\, ds, \quad t\geq 0
$$
is a martingale w.r.t. the natural filtration for $X$. A  {martingale problem} $(L,\mathcal{D})$ is said to be \emph{well posed} in $D(\ax)$ (the Skorokhod space of c\`adl\`ag functions), if for any probability measure $\pi$ on $\Re$ there exists a solution $X$ to this problem with c\`adl\`ag trajectories and $\mathrm{Law}(X_0)=\pi$, and for any two such solutions their distributions in $D(\ax)$ coincide. By the second definition, {L\'evy-type process} associated to $L$ is a solution to the MP $(L, \mathcal{D})$ with $L$ given by \eqref{L_gen} and $\mathcal{D}=C_\infty^2$ or $C_0^2$.

Arbitrary L\'evy process $X$ satisfies both of the above definitions; the corresponding operator $L$ is defined by \eqref{L_gen} with
$b(x)\equiv b, a(x)\equiv a, \mu(x, \cdot )\equiv \mu$, where  $(b,a,\mu)$ is the characteristic triplet for $X$. This explains the name  \emph{L\'evy-type process}, which  we use systematically. The principal problem (I) outlined in the Introduction can be now formulated precisely: given a triplet $b(x), a(x), \mu(x, du)$, is    a L\'evy-type process associated to $L$ uniquely defined in either/both of two ways explained above? That is, does there exist a unique Feller process with the prescribed restriction of the generator, or/and is the MP $(L, \mathcal{D})$ well posed?   The problem (II) then would be to describe -- in the most explicit way it is possible -- of the transition probability $P_t(x, dy)$ of the process $X$.

We will study these two questions in the particular setting of \emph{locally $\alpha$-stable} L\'evy-type operators/ processes,
which we now introduce. A real-valued  $\alpha$-stable process is a L\'evy process which lacks the diffusion term ($a=0$),  may contain a non-trivial shift ($b\not=0$),
and has the L\'evy measure
$$
\mu(du)=\mu^{(\alpha; \lambda,\rho)}(du):=\lambda{1+\rho\, \sgn\, u\over |u|^{\alpha+1}}\,du.
$$
 Taking the intensity and skewness parameters state dependent, $\lambda:\Re\to \Re^+$, $\rho:\Re\to [-1,1]$, we obtain \emph{an $\alpha$-stable L\'evy kernel}
$$
\mu^{(\alpha)}(x; du):=\mu^{(\alpha; \lambda(x),\rho(x))}(du)=\lambda(x){1+\rho(x)\, \sgn\, u\over |u|^{\alpha+1}}\,du.
$$
Our actual L\'evy kernel has the form
\be\label{mu_decomp}
\mu(x; du)=\mu^{(\alpha)}(x; du)+\nu(x; du);
\ee
that is, it is a perturbation of an $\alpha$-stable  kernel by a certain `residual'
kernel $\nu(x; du)$. The residual kernel $\nu(x; du)$ is allowed to be signed, and we denote  by
$\nu_+(x; du)$, $\nu_-(x; du)$ the positive (resp. the negative) parts of its  Hahn decomposition $\nu(x; du)=\nu_+(x; du)-\nu_-(x; du)$.
%If $\nu_-(x; du)$ is non-trivial, $\mu(x;du)$ is actually a truncation of the  $\alpha$-stable L\'evy kernel $\mu^{(\alpha)}(x;du)$.
The negative part  $\nu_-(x; du)$ is assumed to be dominated by $\mu^{(\alpha)}(x; du),$ and $|\nu|(x;du)=\nu_+(x; du)+\nu_-(x; du)$ (the variation of $\nu(x; du)$) is assumed to be a {L\'evy kernel}. The main assumption imposed on the residual kernel is that, uniformly in $x$, the Blumenthal-Getoor \emph{activity index} for $|\nu|$ is strictly smaller than $\alpha$; that is, for some $\beta<\alpha$
\be\label{cond_ups_small}
|\nu| (x; \{|u|>r\}) \leq Cr^{-\beta}, \quad r\in (0, 1].
\ee
Since the Blumenthal-Getoor index for an $\alpha$-stable L\'evy measure equals $\alpha$,
 this condition  actually  means that the small jump behavior of $\mu(x;du)$ is asymptotically the same as for its $\alpha$-stable  part $\mu^{(\alpha)}(x;du)$, and this is our  reason  to call the kernel \eqref{mu_decomp} \emph{locally $\alpha$-stable}.

Summarizing all the above, we specify the \emph{locally $\alpha$-stable} L\'evy-type operator as an operator of the form \eqref{L_gen} with $\mu(x, du)$ given by  \eqref{mu_decomp}, $a(x)\equiv 0$, and possibly non-trivial $b(x)$; that is,
\be\label{L_gen1}
Lf(x)=b(x)f'(x)+\int_\Re\Big(f(x+u)-f(x)-uf'(x)1_{|u|\leq 1}\Big)\Big(\mu^{(\alpha)}(x; du)+\nu(x; du)\Big).
\ee

\section{The main results}\label{s3} In this section we specify the conditions imposed on the model,  formulate the main results, and make a discussion which includes examples, possible extensions, and related references.

\subsection{Conditions}
In what follows, $L$ is the L\'evy-type operator  defined by  \eqref{L_gen1}, and  \eqref{cond_ups_small} is assumed. Throughout the paper we denote by  $C$  a generic constant whose particular value may vary from place to place. We define the \emph{compensated drift coefficient} by
$$
\wt b(x)=  b(x)-1_{\alpha<1}\int_{|u|\leq 1}u\, \mu^{(\alpha)}(x; du)-1_{\beta<1}\int_{|u|\leq 1}u\, \nu(x; du),
$$
and assume the following.

\begin{itemize}
\item[$\mathbf{H}^{drift}.$] (On the compensated drift coefficient). There exists index $\eta\in[0,1]$ satisfying the \emph{balance condition}
\be\label{balance}
\alpha+\eta>1,
\ee
such that
\be\label{b_H_lin}
|\wt b(x)-\wt b(y)|\leq C|x-y|^\eta,  \quad |x-y|\leq 1.
\ee

  \item[$\mathbf{H}^{(\alpha)}.$] (On coefficients $\lambda, \rho$ of the kernel $\mu^\alpha$).
\begin{itemize}
  \item[(i)] $\lambda, \rho$ are H\"older continuous with some index $\zeta\in(0, \alpha)$;
  \item[(ii)] for some $0<\lambda_{\min}<\lambda_{\max}$,
  $$\lambda_{\min}\leq \lambda(x)\leq \lambda_{\max}.
  $$
\end{itemize}

  \item[$\mathbf{H}^{\nu}.$](On the residual kernel $\nu$).  We deal with  two types of upper bounds:
\begin{itemize}
  \item[(i)](weak bound) the kernel $\nu(x,du)$ satisfies \eqref{cond_ups_small} and the following  `tail condition':
     \be\label{cond_ups_large}
 \sup_{x\in \Re}|\nu| (x, \{|u|\geq R\})\to 0, \quad R\to \infty;
\ee
  \item[(ii)](strong bound) the kernel has the density $$
  q_\nu(x,u)={\nu(x,du)\over du},
  $$
 which satisfies
\be\label{reg_cond}
  |q_\nu(x,u)|\leq C|u|^{-\beta-1}1_{|u|\leq 1}+C|u|^{-\gamma-1}1_{|u|>1}
\ee
  with some $ \beta\in (0,\alpha), \gamma>0$.
\end{itemize}

  \item[$\mathbf{H}^{cont}.$] (Continuity assumptions). The kernel $\nu(x,du)$ is assumed to have the following weak continuity property: for any $f\in C(\Re)$ with compact support in $\Re\setminus\{0\}$, the function
\be\label{cont}
x\mapsto \int_{\Re}f(u)\nu(x,du)
\ee
is continuous. The drift coefficient $b$ is assumed to be continuous.
\end{itemize}

  Note that, thanks to condition $\mathbf{H}^{(\alpha)},$ the continuity of \eqref{cont} yields similar continuity for the entire kernel $\mu(x,du)=\mu^{(\alpha)}(x,du)+\nu(x,du)$.

\begin{rem}\label{r30} In the \emph{super-critical regime} $\alpha<1$, the balance condition \eqref{balance} is close to the necessary one for the process to be well defined. This observation dates back to \cite{TTW74}, where a natural  example of an SDE driven by a symmetric additive $\alpha$-stable noise with $\eta$-H\"older continuous $b$ is given, which has two different weak solutions. We emphasise that in the current setting the balance condition involves  the compensated drift coefficient $\wt b$ instead of the  original  $b$. \end{rem}

\begin{rem}\label{r32}
 A good way to understand the role of the continuity condition $\mathbf{H}^{cont}$ is to observe that, if (say) $\nu\equiv 0$ and $b$ is discontinuous, it is impossible for the  operator \eqref{L_gen1} that $Lf$ is continuous for all $f\in C_0^2$, and thus the first definition of the L\'evy-type process becomes inappropriate. This complication is of a technical kind, which is not related to our main goal to derive representation \eqref{p_rep} for the transition probability of the process.  Thus we adopt $\mathbf{H}^{cont}$ and  avoid further technical complications.
\end{rem}

\subsection{The main statements}

Our first main result uniquely identifies a locally $\alpha$-stable  L\'evy type process with given characteristics.

\begin{thm}\label{mainthm1}  Let $L$ be given by \eqref{L_gen1} and conditions $\mathbf{H}^{drift}$, $\mathbf{H}^{(\alpha)}$, $\mathbf{H}^{\nu}$(i), and $\mathbf{H}^{cont}$  hold true.
Then the martingale problem $(L, C_0^2)$ is well posed in $D(\ax)$ and, at the same time, the solution $X$ of this martingale problem is the unique Feller process, whose generator $A$ restricted to $C_0^\infty$ coincides with $L$. This process is strong Feller and possesses a transition probability density $p_t(x,y)$.
\end{thm}

Next, we provide several versions of the representation \eqref{p_rep} with different types of bounds on the residual kernel $R_t(x,y)$,  depending on the actual assumptions imposed in the characteristics of the process. Following the two-stage scheme outlined in the Introduction, we first do this under the basic set of conditions   $\mathbf{H}^{drift}$, $\mathbf{H}^{(\alpha)}$, $\mathbf{H}^{\nu}$ (i), an then discuss modifications  under additional assumptions.

Let us introduce more notation. By $g^{(\lambda,\rho,\upsilon)}(w)$ we denote the density of the $\alpha$-stable distribution  with the intensity $\lambda$, skewness $\rho$, and a shift $\upsilon$:
\be\label{stable_den}
g^{(\lambda,\rho,\upsilon)}(w)={1\over 2\pi}\int_\Re e^{-iw\xi+\Psi^{(\lambda,\rho,\upsilon)}_\alpha(\xi)}\, d\xi,
\ee
\be\label{stable_psi}
 \Psi^{(\lambda,\rho, \upsilon)}_\alpha(\xi)= i\xi\upsilon+\int_{\Re}\Big(e^{iu\xi}-1-iu\xi 1_{|u|\leq 1}\Big)\mu^{(\alpha,\lambda,\rho)}(du).
\ee
Next, we denote
$$
\delta_\eta={\eta+\alpha-1\over \alpha}>0,\quad  \delta_\zeta={\zeta\over \alpha}>0, \quad \delta_\beta={\alpha-\beta\over \alpha}>0, \quad \delta_{\eta,\zeta,\beta}=\min(\delta_\eta,\delta_\zeta,\delta_\beta);
$$
note that the positivity of $\delta_\eta$ is just the balance condition \eqref{balance}. We fix (arbitrary) positive $\delta< \delta_{\eta,\zeta,\beta}.$ We also fix (arbitrary) $T>0$ and furthermore consider $t\leq T$, only. Denote
\be\label{m_tb_t}
m_t^{\mu}(x)=\int_{t^{1/\alpha<|u|\leq 1}}u\mu(x,du),\quad b_t(x)=b(x)-m_t^{\mu}(x),
\ee
the \emph{partial compensator} of the kernel \eqref{mu_decomp} with the truncation level $t^{1/\alpha}$, and \emph{partially compensated} drift coefficient, respectively. Define the corresponding \emph{mollified} coefficient
$$
 B_t(x)=\int_{\Re}b_t(x-z){1\over 2\sqrt{\pi} t^{1/\alpha}}e^{-z^2t^{-2/\alpha}}d z.
$$
This coefficient is chosen in such a way that
\be\label{b-B}
\sup_x|b_t(x)-B_t(x)|\leq Ct^{-1+1/\alpha+\delta},
\ee
\be\label{Lip_B}
\mathrm{Lip}\,(B_t)=\sup_{x\not=y}{|B_t(x)-B_t(y)|\over |x-y|}\leq Ct^{-1+\delta},
\ee
see Appendix \ref{sA1} (recall that $\delta<\delta_\beta<1$). We define $\chi_s(x), s\geq 0, x\in \Re$ as the solution to the Cauchy problem
\be\label{chi}
{d\over ds}\chi_s(x)= B_{s}(\chi_s(x)), \quad \chi_0(x)=x.
\ee
Note that by \eqref{Lip_B} the family of Lipschitz constants $\mathrm{Lip}\,(B_t), t>0$ is integrable on any finite segment,
thus $\chi_t(x)$ is uniquely defined by the classical Picard successful approximation procedure. We define
$$
\lambda_t(x)={1\over t}\int_0^t\lambda(\chi_s(x))\, ds, \quad \rho_t(x)={1\over t\lambda_t(x)}\int_0^t\lambda( \chi_s(x))\rho( \chi_s(x))\, ds;
$$
that is, $\lambda_t(x)$ and $\lambda_t(x)\rho_t(x)$ are the  averages of the functions $\lambda(\cdot)$, $\lambda(\cdot)\rho(\cdot)$ along the trajectory  $\chi_\cdot(x)$ on the segment $[0,t]$.
We also denote $\upsilon(x)=2\lambda(x)\rho(x)$,
\be\label{W-def}
W_\alpha(t;s)=t^{-1/\alpha}\int_{s^{1/\alpha}}^{ t^{1/\alpha}}\,{dr\over r^{\alpha}}=\left\{
                                                                                       \begin{array}{ll}
                                                                                         {1\over 1-\alpha}  t^{-1/\alpha} (t^{1/\alpha-1}-s^{1/\alpha-1}), & \alpha\not=1, \\
                                                                                           t^{-1}  (\log t-\log s), & \alpha=1,
                                                                                       \end{array}
                                                                                     \right.
\quad 0\leq s\leq t,
\ee
and put
$$
\upsilon_t(x)=\int_0^t\upsilon(\chi_s(x))W_\alpha(t;s)\, ds.
$$
Note that
\be\label{W_prob}
\int_0^tW_\alpha(t,s)\, ds=1;
\ee
that is,  $\upsilon_t(x)$ is also an average of $\upsilon(\cdot)$ along the trajectory   $\chi_\cdot(x)$, but with respect to a certain (non-uniform) probability distribution on $[0,t]$. We finally define
\be\label{g}
g^{t,x}(w)=g^{(\lambda_t(x), \rho_t(x), \upsilon_t(x))}(w),
\ee
the $\alpha$-stable density with the `$\chi$-averaged' parameters  $\lambda_t(x), \rho_t(x), \upsilon_t(x)$ defined above.

Now we are ready to state our second main result. Recall that we consider $t\in [0,T]$, where $T$ is arbitrary but fixed; the particular values of the constants $C$ below may depend on $T$ and particular choice of $\delta< \delta_{\eta,\zeta,\beta}.$

\begin{thm}\label{mainthm2} I. Let conditions $\mathbf{H}^{drift}$, $\mathbf{H}^{(\alpha)}$, $\mathbf{H}^{\nu}$(i), and $\mathbf{H}^{cont}$  hold true. Then
\be\label{p_rep1}
p_t(x,y)={1\over t^{1/\alpha}}g^{t,x}\left({y-\chi_t(x)\over t^{1/\alpha}}\right)+R_t(x,y),
\ee
where
\be\label{R_bound_int}
\sup_x\int_{\Re}|R_t(x,y)|\, dy\leq Ct^\delta.
\ee

 II. Assume in addition that for some $\delta_\nu>0$
\be\label{invert}
\sup_{w\in \Re}\left|t^{-1/\alpha}\int_\Re \nu\Big(x; \big\{|u|>t^{1/\alpha},|x+u-w|\leq t^{1/\alpha}\big\}\Big)\, dx\right|\leq Ct^{-1+\delta_\nu}.
\ee

Then
\be\label{R_bound_unif}
\sup_{x,y}|R_t(x,y)|\leq Ct^{-1/\alpha+\delta_\infty}, \quad \delta_\infty=\min(\delta,  \delta_\nu).
\ee

\end{thm}

Our last main result provides a point-wise kernel estimate for the residual term $R_t(x,y)$ under the stronger assumption $\mathbf{H}^{\nu}$ (ii). Denote
\be\label{abg}
G^{(\alpha, \beta, \gamma)}_t(x,y)=    \left\{                                 \begin{array}{ll}{t^{-1/\alpha}},& |y-x|\leq (t^{1/\alpha}\wedge 1),\\
                                    {t^{\beta/\alpha}} |y-x|^{-\beta-1}, & (t^{1/\alpha}\wedge 1)<|y-x|\leq 1,\\
                                 {t^{\beta/\alpha}} |y-x|^{-\gamma-1}, & |y-x|>1.
                                    \end{array}
                                  \right.
\ee

\begin{thm}\label{mainthm3} Let conditions $\mathbf{H}^{drift}$, $\mathbf{H}^{(\alpha)}$, $\mathbf{H}^{\nu}$ (ii), and $\mathbf{H}^{cont}$  hold true. Then
\be\label{R_bound_point}
|R_t(x,y)|\leq Ct^\delta G^{(\alpha, \alpha, \alpha)}_t(\chi_t(x),y)+Ct^{\delta'}G^{(\alpha, \beta', \gamma')}_t(\chi_t(x),y)
\ee
with
\be\label{primes}
\beta'=\max(\beta, \alpha- \zeta),\quad \gamma'=\min(\alpha, \gamma), \quad \delta'={\alpha-\beta'\over \alpha}>0.
\ee
\end{thm}

As a direct corollary, we get an upper bound for the entire transition probability density $p_t(x,y)$. Denote
$
G^{(\alpha)}(x)=|x|^{-\alpha-1}\wedge 1.
$
\begin{cor}\label{cor_p} Let conditions $\mathbf{H}^{drift}$, $\mathbf{H}^{(\alpha)}$, $\mathbf{H}^{\nu}$ (ii), and $\mathbf{H}^{cont}$  hold true. Then
\be\label{p_upper bound}
p_t(x,y)\leq {C\over t^{1/\alpha}}G^{(\alpha)}\left({y-\chi_t(x)\over t^{1/\alpha}}\right)+Ct|y-\chi_t(x)|^{-\gamma'-1}1_{|y-\chi_t(x)|>1}, \quad t\in (0, T].
\ee
\end{cor}

The following two examples show that there is a the substantial difference between three types of the estimates given above: (i)  integral-in-$y$ (Theorem \ref{mainthm2}, I); (ii)   uniform-in-$(x,y)$ (Theorem \ref{mainthm2}, II); (iii) kernel (Theorem \ref{mainthm3}). The first example shows that, for singular kernels $\mu(x;du)$, the estimates (ii), (iii) may simply fail.

\begin{ex}\label{ex1}  Let the `nuisance part' of the noise correspond to the possibility of the process $X_t$ to jump, at  Poisson time instants, to the point $0$; that is, $
\nu(x,du)=\delta_{-x}(du).$
 Then
$$
p_t(x,y)\geq e^{-1}\int_0^tp_{t-s}^{(\alpha)}(0,y)\, ds,
$$
where $p_{t}^{(\alpha)}(x,y)$ denotes the transition probability density for the process with the kernel $\mu^{(\alpha)}(x,du)$. If $b\equiv 0$, $\lambda\equiv 1, \rho\equiv 0$, then
$$
p_{t}^{(\alpha)}(x,y)\asymp t^{-1/\alpha}G^{(\alpha)}\left({y-x\over t^{1/\alpha}}\right),
$$
and thus for $\alpha\leq 1$ the function $p_t(x,y)$ is unbounded at the vicinity of the point $y=0$.
\end{ex}

The difference between the kernel and uniform-in-$(x,y)$ estimates is more subtle. Of course, the kernel estimates yield both the integral-in-$y$ and uniform-in-$(x,y)$ estimates, but the cost is that the (strong) condition $\mathbf{H}^{\nu}$ (ii) is needed, which in particular requires $\nu(x,du)$ to be smooth. This may be too restrictive when a model with a microstructural residual noise in the spirit of \cite{AJ07} is considered. Our second example shows that the additional assumption \eqref{invert}, which guarantees uniform-in-$(x,y)$ bounds, is substantially  weaker than $\mathbf{H}^{\nu}$ (ii), and  can hold true for  singular nuisance kernels.

\begin{ex}\label{ex2} Let $\nu(x,du)=\nu(du)$, then simply by the Fubini theorem and  \eqref{cond_ups_small} we have
$$
\ba
&\left|t^{-1/\alpha}\int_\Re \nu\Big(x; \big\{|u|>t^{1/\alpha},|x+u-w|\leq t^{1/\alpha}\big\}\Big)\, dx\right|
\\&\hspace*{2cm}\leq \int_{|u|>t^{1/\alpha}} \left(t^{-1/\alpha}\int_{|x+u-w|\leq t^{1/\alpha}}\, dx\right)\,|\nu|(du)\leq Ct^{-\beta/\alpha}=Ct^{-1+\delta_\beta}.
\ea
$$
More generally, let $\nu(x;du)$ possess a bound
$$
|\nu|(x,du)\leq \nu'(v:c(x,v)\in du),
$$
where $\nu'$ is a L\'evy measure satisfying \eqref{cond_ups_small} and $c(x,u)$ satisfies $|c(x,u)|\leq C|u|$, for each $u$ the function
$x+c(x,u)$  is $C^1$ and is invertible (in $x$), and
\be\label{invert_C1}
|1+c'_x(x,u)|^{-1}\leq C.
\ee
Then we can obtain  \eqref{invert} first changing the variables $x'=x+c(x,u)$ and then using the Fubini theorem and \eqref{cond_ups_small} in the same way as above.
\end{ex}

\subsection{SDEs} For the reader's convenience, we formulate separately the version of the above  results in the case where the process $X$ is a solution to an SDE. Consider the SDE
\be\label{SDE2}
d X_t=b(X_t)\,d t+\sigma(X_{t-})\, dZ_t+\int_{|u|\leq 1}c(X_{t-}, u)\wt N(d t, du)+\int_{|u|> 1}c(X_{t-}, u)N(d t, du)
\ee
where $Z$ is an $\alpha$-stable process, $N(dt, du)$ is an independent of $Z$ Poisson point measure with the compensator $dt\nu'(du)$, and  $N(dt, du)=N(dt, du)-dt\nu'(du)$ is the corresponding martingale measure. Assume that $Z$ has the characteristic triplet $(0, 0, \mu^{(\alpha;  \lambda, \rho)})$ and $|c(x,u)|\leq C|u|$.
Denote
$$
\wt b(x)=  b(x)-1_{\alpha<1}{2\lambda\rho\over 1-\alpha}\sigma(x)-1_{\beta<1}\int_{|u|\leq 1}c(x,u)\, \nu'(du).
$$
\begin{prop} Let the following assumptions hold:
\begin{itemize}
  \item $\wt b$ satisfies $\mathbf{H}^{drift}$;
  \item $\sigma$ is $\zeta$-H\"older continuous and for some $c_1, c_2>0$
  $$
   c_1\leq \sigma(x)\leq c_2;
   $$
  \item for some $\beta<\alpha$,
 $$
 \nu'(x; \{|u|>r\}) \leq Cr^{-\beta}, \quad r\in (0, 1];
$$
\item the functions $b(x)$ and  $x\mapsto c(x,\cdot)\in L_1((u^2\wedge 1)\nu'(du))$ are continuous.
\end{itemize}

Then the SDE \eqref{SDE2} has unique weak solution $X$, and this solution is a strong Feller Markov process. The transition probability of this process has a density $p_t(x,y)$ which has representation \eqref{p_rep}, where
\begin{itemize}
  \item the regressor $\mathfrak{f}_t(x)=\chi_t(x)$ is defined by \eqref{chi} with $B_t(x)$ which corresponds to
  $$
  b_t(x)=b(x)-2\lambda\rho\sigma(x)\int_{t^{1/\alpha}}^1{du\over u^\alpha}-\int_{t^{1/\alpha}<|u|\leq 1}c(x,u)\, \nu'(du)
  $$
  \item the density of the $\alpha$-stable innovation term  has the form
    $g^{t,x}(w)=g^{(\lambda_t(x), \rho, \upsilon_t(x))}(w)$ with
        $$
\lambda_t(x)={\lambda\over t}\int_0^t\sigma(\chi_s(x))^\alpha\, ds, \quad \upsilon_t(x)=2\lambda\rho\int_0^t\sigma(\chi_s(x))^\alpha W_\alpha(t;s)\, ds;
$$
  \item the residual term $R_t(x,y)$ satisfies \eqref{R_bound_int}.
 \end{itemize}
  In addition,
  \begin{itemize}
      \item if the function $x+c(x,u)$  is $C^1$, is invertible in $x$ and \eqref{invert_C1} holds, then the residual term $R_t(x,y)$ satisfies \eqref{R_bound_unif};
  \item  if
  $$
  {\nu'(du)\over du}=C|u|^{-\beta-1}1_{|u|\leq 1}+C|u|^{-\gamma-1}1_{|u|>1},
  $$
  $c(x, \cdot)\in C^1$ and
  $$\inf_{x,u}|c'_u(x,u)|>0,$$
  then \eqref{R_bound_point} holds.
\end{itemize}
\end{prop}

The L\'evy-type operator, which formally corresponds to the SDE \eqref{SDE2} is given by
$$\ba
Lf(x)=b(x) f'(x)&+\int_\Re\Big(f(x+\sigma(x)u)-f(x)-\sigma(x)uf'(x)1_{|u|\leq 1}\Big)\mu^{(\alpha; \lambda,\rho)}(du)
\\&+\int_\Re\Big(f(x+c(x,u)-f(x)-c(x,u)u f'(x)1_{|u|\leq 1}\Big)\nu'(du).
\ea
$$
Then the uniqueness of the weak solution to the SDE is close to the well posedness of the MP $(L, C_0^\infty)$; for a (simple) formal argument which connects these two notions  see e.g. \cite[Section 4.3]{Ku18}. Thus the required statements follow from Theorems \ref{mainthm1} -- \ref{mainthm3} by simple re-arrangements.

\subsection{Possible extensions} Let us briefly discuss several possible  modifications and extensions of the main results. First, let us note that the case of state-dependent $\alpha=\alpha(x)$ can be treated similarly, but with a more sophisticated and less transparent estimates. We postpone its study to the companion paper \cite{KKS18}, where the multidimensional locally $\alpha$-stable model is considered in the widest possible generality. It is also visible that the sensitivities (i.e. derivatives) of $p_t(x,y)$ w.r.t. $t$ and external parameters can be treated with the same method; in particular we refer to \cite{KK15}, \cite{Ku18} for representations and bounds for $\prt_tp_t(x,y)$ and to \cite{GKK} for an application of such bounds in the accuracy bounds for approximation of integral functionals.  In order not to overextend the exposition, in the current paper we do not address the sensitivities, leaving their study to a further research.

Next, let us mention that the particular form  of the conditionally $\alpha$-stable approximation \eqref{p_rep} obtained in Theorem \ref{mainthm2} is not the only possible one. Namely, one can change consistently the regressor $\mathfrak{f}_t(x)=\chi_t(x)$ and the {$\alpha$-stable innovation term}, providing the following alternative representation, which may be more convenient e.g for  simulation purposes. Define for a given $t>0$ the family $\ov \chi_s^t(x)$, $s\in[0, t]$, $x\in \Re$ as the solution to the Cauchy problem
$$
{d\over ds}\ov\chi_s^t(x)=B_{t}(\ov \chi_s^t(x)), \quad \ov\chi_0^t(x)=x,
$$
and put
$$
\ov\lambda_t(x)={1\over t}\int_0^t\lambda(\ov\chi_s^t(x))\, ds, \quad \ov\rho_t(x)={1\over t\lambda_t(x)}\int_0^t\lambda( \ov\chi_s^t(x))\rho( \ov \chi_s^t(x))\, ds,
$$
$$
\ov g^{t,x}(w)=g^{(\ov\lambda_t(x), \ov\rho_t(x), 0)}(w).
$$
\begin{prop} Let conditions $\mathbf{H}^{drift}$, $\mathbf{H}^{(\alpha)}$, $\mathbf{H}^{\nu}$(i), and $\mathbf{H}^{cont}$  hold true. Then
\be\label{p_rep1ov}
p_t(x,y)={1\over t^{1/\alpha}}\ov g^{t,x}\left({y-\ov \chi_t^t(x)\over t^{1/\alpha}}\right)+\ov R_t(x,y),
\ee
where $\ov R_t(x,y)$ satisfies \eqref{R_bound_int}. Under the additional condition \eqref{invert}  $\ov R_t(x,y)$ satisfies \eqref{R_bound_unif}, and under the condition $\mathbf{H}^{\nu}$ \textrm{(ii)} the term  $\ov R_t(x,y)$ satisfies \eqref{R_bound_point}. In the latter case, $\chi_t(x)$ in the right hand side of
\eqref{R_bound_point} can be replaced by $\ov \chi_t^t(x)$.
\end{prop}
\begin{proof}[Sketch of the proof.] It is clear from the definition of the   density $\ov g^{t,x}$ that
$$g^{(\ov\lambda_t(x), \ov\rho_t(x), \ov\upsilon_t(x))}(w)=\ov g^{t,x}(w-\ov \upsilon_t(x)),
$$where we denote 
$$
\ov\upsilon_t(x)=\int_0^t\upsilon(\chi_s^t(x))W_\alpha(t;s)\, ds.
$$
   On the other hand, one can show similarly to \eqref{chi-chi} that
\be\label{chis}
\left|\ov\chi_s^t(x)- \chi_s(x)\right|\leq Ct^{1/\alpha}, \quad s\leq t, \quad \hbox{ and }\quad  \left|\big(\ov \chi_t^t(x)-t^{1/\alpha}\ov\upsilon_t(x)\big)- \chi_t(x)\right|\leq Ct^{1/\alpha+\delta}.
\ee It follows from the H\"older continuity of $\lambda, \rho, \upsilon$ and the first inequality in \eqref{chis} that 
$$
|\ov\lambda_t(x)-\lambda_t(x)|+|\ov\rho_t(x)-\rho_t(x)|+|\ov\upsilon_t(x)-\upsilon_t(x)|\leq Ct^{\zeta/\alpha}\leq Ct^\delta.
$$
Then 
  the required bounds for
$$\ba
\ov R_t(x,y)&=R_t(x,y)
\\& +{1\over t^{1/\alpha}}g^{(\lambda_t(x), \rho_t(x), \upsilon_t(x))}\left({y-\chi_t(x)\over t^{1/\alpha}}\right)-{1\over t^{1/\alpha}}g^{(\lambda_t(x), \ov\rho_t(x), \ov\upsilon_t(x))}\left({y-\ov \chi_t^t(x)+t^{1/\alpha}\ov\upsilon_t(x)\over t^{1/\alpha}}\right)
\ea
$$
follow by respective bounds for $R_t(x,y)$ and the basic properties of  stable densities (e.g. \eqref{g_bound_x}).
\end{proof}
In the above representations, we define the regressor as the solution to the ODE driven by the (mollified) partially compensated drift, and then determine the parameters of the $\alpha$-stable density of the innovation term by averaging of the correspondent space dependent parameters of the model w.r.t. the solution to the ODE  on the time interval $[0,t]$. These principal components can be further simplified by the cost of making the bounds less precise and (possibly) under additional assumptions. First, let us mention briefly that the \emph{true} solution $\chi_t$  to \eqref{chi} can be replaced by its $k$-th iteration $\chi_t^{(k)}$ in the Picard approximation procedure. The situation here is similar to the one studied in \cite[Section 2.2]{Ku18}, thus we omit a detailed discussion and just mention that for such an approximation to be successful one needs
$$
1+\eta+\dots+\eta^k>{1\over \alpha}.
$$
In particular, the  naive choice of the regressor $\mathfrak{f}_t(x)=x+b(x)t$ mentioned in Introduction corresponds to the case $k=1$. That is, for such a choice to be successful it is required that $\alpha>(1+\eta)^{-1}$, which in particular excludes small values $\alpha\leq 1/2$.

Next, in the case  of \emph{bounded}  $\wt b$,  the innovation term can be further simplified. Namely, in this case  it is easy to verify that
$$
|\chi_t(x)-x|\leq C\Big(t+t^{1/\alpha}\Big)
$$
if $\alpha\not=1$ (in the exceptional case $\alpha=1$ an additional logarithmic term should appear).
 Since $\lambda$ is $\zeta$-H\"older continuous, this yields
$$
|\lambda_t(x)-\lambda(x)|\leq C\Big(t^\zeta+t^{\delta_{\zeta}}\Big),
$$
and the similar bounds hold true for $\rho_t, \upsilon_t, \ov\lambda_t, \ov\rho_t, \ov\upsilon_t$. Then essentially the same argument as in the proof of \eqref{R_tilde_bound} (see Appendix \ref{sA41}) makes it possible to deduce representations
\be\label{p_repbound}
p_t(x,y)={1\over t^{1/\alpha}}g^{x}\left({y-\chi_t(x)\over t^{1/\alpha}}\right)+R_t^{frozen}(x,y)={1\over t^{1/\alpha}}\ov g^{x}\left({y-\ov \chi_t(x)\over t^{1/\alpha}}\right)+\ov R_t^{frozen}(x,y)
\ee
with the $\alpha$-stable densities
$$
g^{x}(w)=g^{(\lambda(x), \rho(x), \upsilon(x))}(w), \quad \ov g^{x}(w)=g^{(\ov\lambda(x), \ov\rho(x), 0)}(w),
$$
which just correspond to the values of the parameters `frozen' at the initial point $x$. The error terms $R_t^{frozen}(x,y), \ov R_t^{frozen}(x,y)$ under the corresponding conditions satisfy analogues of \eqref{R_bound_int}, \eqref{R_bound_unif}, and \eqref{R_bound_point} with $\delta$ changed to $\delta\wedge \zeta$. Note that $\delta<\zeta/\alpha$; that is, for $\alpha\geq 1$ the bounds actually remain unchanged.

\subsection{Some related results}\label{over} We do not give a wide overview of the related results in this extensively developing domain, referring an interested reader to \cite{KK15}, \cite{Kuh17-2}, and a survey paper \cite{KKK} for such reviews. Instead, we focus on a discussion of references directly related to the particular issues treated in the current paper.

\emph{1. Various types of estimates.}
We have already mentioned that the most attention in the available literature  is devoted to kernel-type estimates, see  detailed surveys in  \cite{KK15}, \cite{KKK}. The separate study of integral-in-$y$ and uniform-in-$(x,y)$ estimates is apparently new; note however the  forthcoming book \cite{Ko18}, Sections 5.4, 5.5, where a systematic treatment is given, which leads to a pair of dual $L_1$-$C_\infty$ estimates. These estimates are of the same spirit with ours; however, one should note note that the additive-in-space  bounds  (see \cite[(5.69)]{Ko18}) adopted there  as the main assumption,  in certain settings,  may become too restrictive. Namely, it will become clear from the proof of uniform-in-$(x,y)$ estimate in Section \ref{s5} below that the main property required for such estimate to hold is the integral-in-$x$ bound \eqref{full} which is actually a `dual' analogue of the `direct' integral-in-$y$ estimate. Example \eqref{ex1} shows that, for singular L\'evy kernels, the `direct'  and the `dual' estimates should be treated separately. On the other hand, the additive structure of
\cite[(5.69)]{Ko18} makes the integral-in-$x$ and the integral-in-$y$ estimates synonymic, which does not allow one to
approach singular L\'evy-type models. In the recent preprint \cite{KulStoRyz18}, another (mixed $L_1$-$C_\infty$) type of estimates is proposed to treat the singular model L\'evy-type based on the  multidimensional SDEs with cylindrical $\alpha$-stable noise and non-trivial rotation.

Let us mention that the $L_1$-approach, based on integral-in-$y$ estimates only, has a deep connection, at least on the level of the principle ideas, with the approach to the well-posedness of the martingale problem for  integro-differential operators which dates back to  \cite{G68} and \cite{Ko84a}, \cite{Ko84b}.

\emph{2. Non-symmetry of the L\'evy noise.} The heat kernel estimates for L\'evy and L\'evy-type processes were mainly studied for symmetric noises; the non-symmetric setting becomes the subject of a study just  in the few last years. The most advanced study in this direction available to the author is given by the recent preprint \cite{S18}; we refer there for an overview of few other recent results in the same direction. In the model from \cite{S18}, the external drift (our $b$) is not included, as well as the nuisance kernel $\nu$. On the other hand, the class of the kernels treated therein is substantially wider than our class of $\alpha$-stable principal parts.

\emph{3. Non-boundedness of the drift coefficient.} It is traditional for the literature exploiting the analytical parametrix-type methods that the coefficients are assumed to be globally bounded. On the other hand, it was specially pointed to the author by H.~Masuda that, for various applications esp. in statistics it is highly desirable for the theory to cover mean reverting models of the Ornstein-Uhlenbeck type. This explains the special attention paid in the paper to the case of unbounded $b$. The only reference known to the author, where such non-boundedness is allowed, is an  apparently yet  not published preprint \cite{H15}.

\section{Preliminaries to the proofs: the parametrix method and an integral representation for $p_t(x,y)$}\label{s4}

In this section we make preparation for the proofs of the main results. We introduce an integral equation whose unique solution $p_t(x,y)$  later on will be proved to be the transition probability density  of the target process $X$.  Such a construction is motivated by the \emph{parametrix method}, which is a classical tool for constructing fundamental solutions  to parabolic Cauchy problems.  We present here only the rigorous step-by-step exposition without additional discussion of the heuristics behind the method; for such a discussion e.g. \cite{KK15}, \cite{Ku18}.

\subsection{The parametrix method: an outline, and the choice of the zero order approximation}\label{s41}
 In this section, we introduce the main objects and explain the  method. We will repeatedly use the following notation for space- and time-space convolutions of functions:
$$
(f\ast g)_t(x,y):=\int_{\Re^d}f_{t}(x,z)g_{t}(z,y)\, dz,\quad
(f\star g)_t(x,y):=\int_0^t\int_{\Re^d}f_{t-s}(x,z)g_{s}(z,y)\, dzds.
$$

We will fix a function $p_t^0(x,y)$,    a `zero order approximation' to the unknown $p_t(x,y)$, which
 will belong to $C^1(0, \infty)$ in $t$ and to $C^2_\infty$ in $x$. In particular, the following `differential error term' will be well defined point-wisely:
\be\label{Phi}
\Phi_t(x,y):=-\Big(\prt_t-L_x\Big)p_t^0(x,y),\quad x,y\in\Re,
\ee
here and below the lower index of an operator indicates the variable at which the operator is applied. Under the proper choice of $p_t^0(x,y)$, the  kernel $\Phi_t(x,y)$ will satisfy
\be\label{Phi_bound}
\sup_{x\in \Re}\int_{\Re}|\Phi_t(x,y)|\, dy\leq Ct^{-1+\delta}.
\ee
The cornerstone of the construction is given by the 2nd type Fredholm integral equation
\be\label{eq_int}
p_t(x,y)=p^0_t(x,y)+(p\star \Phi)_t(x,y),
\ee
which we interpret in the following way. With the {time horizon} $T>0$ being fixed, consider the Banach space of the kernels $\Upsilon_t(x,y)$ on
$[0,T]\times\Re\times \Re$ with the norm
$$
 \|\Upsilon\|_{\infty,1,1}=\sup_{x\in \Re}\int_0^T\int_\Re|\Upsilon_t(x,y)|\, dydt.
$$
Consider also the Banach space $L_{\infty,\infty,1}^T$  of  functions $f_t(x,y)$ with the norm
$$
\|f\|_{\infty,\infty,1}=\sup_{x\in \Re,t\in [0,T]}\int_\Re|f_t(x,y)|\, dy.
$$
 Any kernel $\Upsilon\in L_{\infty,1,1}^T$ generates a bounded linear operator in $L_{\infty,\infty,1}^T$
$$
(A^\Upsilon f)_t(x,y)=(f\star \Upsilon)_t(x,y),
$$
 with the operator norm of $A^\Upsilon$ bounded by $\|\Upsilon\|_{\infty,1,1}$. By \eqref{Phi_bound}, the kernel $\Phi_t(x,y)$ belongs to $L_{\infty,1,1}^T.$  Then we naturally  interpret \eqref{eq_int} as an  equation
\be\label{eq_op}
p=p^0+A^\Phi p
\ee
in the Banach space $L_{\infty,\infty,1}^T.$ It is an easy calculation that \eqref{Phi_bound} yields
\be\label{Phi_bound_k}
\sup_x \int_{\Re} |\Phi^{\star k}_t(x,y)|\, dy\leq t^{-1+k\delta}{C^k\Gamma(\delta)^k\over \Gamma(k\delta)},\quad \Phi^{\star k}=\underbrace{\Phi\star\dots\star\Phi}_{k}, \quad k\geq 1,
\ee
see Section \ref{s44} below. Then
$$
\sum_{k=1}^\infty\|(A^\Phi)^k\|=\sum_{k=1}^\infty\|A^{\Phi^{\star k}}\|\leq \sum_{k=1}^\infty\|\Phi^{\star k}\|_{\infty,1,1}\leq \sum_{k=1}^\infty
T^{k\delta}{C^k\Gamma(\delta)^k\over \Gamma((k+1)\delta)}<\infty,
$$
and therefore the solution to the  equation \eqref{eq_op} in $L^T_{\infty, \infty,1}$ is uniquely specified by the classical von Neumann series representation:
\be\label{sol_1}
p_t(x,y)=p_t^0(x,y)+\sum_{k\geq 1}(p^0\star\Phi^{\star k})_t(x,y)=p_t^0(x,y)+(p^0\star\Psi)_t(x,y),
\ee
\be\label{Psi}
\Psi_t(x,y)=\sum_{k\geq 1}\Phi^{\star k}_t(x,y),
\ee
with the series convergent in $L_{\infty,\infty,1}^T$ and $L_{\infty,1,1}^T,$ respectively.

Now, let us proceed with specification of the  zero-order approximation $p_t^0(x,y)$ for our particular model.
We define the function $\kappa_s(y), s\geq 0,  y\in \Re$ as the solution to the Cauchy problem
$$
{d\over ds}\kappa_s(y)=-B_{s}(\kappa_s(y)),\quad s\geq 0, \quad \kappa_0(y)=y, \quad y\in \Re.
$$
Define for $z\in \Re$, $t>0$
$$
\Psi_\alpha(t,z;\xi)=\int_0^t \int_{\Re}\Big(e^{iu\xi}-1-iu\xi 1_{|u|\leq s^{1/\alpha}}\Big)\mu^{(\alpha)}(\kappa_s(z);du)\, ds,
$$
which   has representation in the form
\be\label{Q1}
\Psi_\alpha(t,z;\xi)=\Psi^{(\wt\lambda_t(z),\wt \rho_t(z), \wt \upsilon_t(z))}_\alpha(t^{1/\alpha}\xi)
\ee
with
$$
\wt \lambda_t(z)={1\over t}\int_0^t\lambda(\kappa_s(z))\, ds, \quad \wt \rho_t(z)={1\over t\wt\lambda_t(z)}\int_0^t\lambda(\kappa_s(z))\rho(\kappa_s(z))\, ds,
$$
$$
\wt \upsilon_t(z)=\int_0^t\upsilon(\kappa_s(z)) W_\alpha(t;s)\, ds,
$$
recall that $W_\alpha(t;s)$ is defined in \eqref{W-def}. We will prove \eqref{Q1} in Appendix \ref{sA41}; this identity actually means that
$\Psi_\alpha(t,z;\xi)$ is a characteristic exponent of an  $\alpha$-stable law.  We denote by
$
h^{t, z}(w)
$
the corresponding $\alpha$-stable distribution density
$$
h^{t, z}(w)={1\over 2\pi}\int_\Re e^{-iw\xi+\Psi_\alpha(t,z;\xi)}\, d\xi,
$$
and define
\be\label{p^0}
p_t^0(x,y)=h^{t, y}(\kappa_t(y)-x).
\ee
Denote
\be\label{wtg}
\wt g^{t,z}(w)= g^{(\wt\lambda_t(z), \wt\rho_t(z), \wt\upsilon_t(z))}(w),
\ee
then by \eqref{Q1} the formula can be written as
\be\label{p^01}
p_t^0(x,y)={1\over t^{1/\alpha}}\wt g^{t,y}\left({\kappa_t(y)-x\over t^{1/\alpha}}\right).
\ee

\subsection{Kernel  $\Phi_t(x,y)$: decomposition and estimates}\label{s42}
Define an auxiliary operator
$$
\Lba^{(\alpha),z,t}f(x)=\int_{\Re}\Big(f(x+u)-f(x)-uf'(x)1_{|u|\leq t^{1/\alpha}}\Big)\mu^{(\alpha)}(\kappa_{t}(z);du), \quad f\in C_\infty^2.
$$
The following identity is crucial for the entire construction.
\be\label{arrow_id}
(\prt_t-\Lba_{x}^{(\alpha),z,t}){h}^{t,z}\left(w-x\right)=0, \quad t>0,\quad  x,w,z\in \Re.
\ee
This identity can be  verified using the formula \eqref{p^0} and a  standard Fouier analysis-based argument; see Appendix \ref{sA41}.   We have
$$\ba
\prt_tp_t^0(x,y)=\prt_t h^{t, y}(\kappa_t(y)-x)&=\prt_t h^{t, y}(w-x)\Big|_{w=\kappa_t(y)}+\prt_w h^{t, y}(w-x)\Big|_{w=\kappa_t(y)}\prt_t\kappa_t(y)
\\&=\prt_t h^{t, y}(w-x)\Big|_{w=\kappa_t(y)}-\prt_x h^{t, y}(w-x)\Big|_{w=\kappa_t(y)}\prt_t\kappa_t(y).
\ea $$
Thus, combining \eqref{arrow_id} and the fact that $\prt_t(\kappa_t(y))=-B_t(\kappa_t(y))$, we get
\be\label{p^0_dt}\ba
\prt_tp_t^0(x,y)&=\Lba_{x}^{(\alpha),y,t} h^{t, y}(\kappa_t(y)-x)+\prt_x h^{t, y}(\kappa_t(y)-x)B_t(\kappa_t(y))
\\&=\Lba_{x}^{(\alpha),y,t}p_t^0(x,y)+B_t(\kappa_t(y))\prt_xp_t^0(x,y).
\ea
\ee
On the other hand, for the operator $L$ defined by \eqref{L_gen1} we have the following decomposition:
\be\label{L_decomp}
L =b\prt_x+L^{(\alpha),x,1}+L^{\nu,x,1} =b_t\prt_x+L^{(\alpha),x,t}+L^{\nu,x,t},
\ee
where
$$
L^{(\alpha),z,t}f(x)=\int_{\Re}\Big(f(x+u)-f(x)-uf'(x)1_{|u|\leq t^{1/\alpha}}\Big)\mu^{(\alpha)}(z;du), \quad f\in C_\infty^2,
$$
\be\label{Lnu}
L^{\nu, z,t}f(x)=\int_{\Re}\Big(f(x+u)-f(x)-uf'(x)1_{|u|\leq t^{1/\alpha}}\Big)\nu(z;du), \quad f\in C_\infty^2.
\ee
Now we can represent $\Phi$ in the following form:
\be\label{Phi_decomp}\ba
\Phi_t(x,y)&=(L_x-\prt_t)p_t^0(x,y)
\\&=\Big(b_t(x)-B_t(\kappa_t(y))\Big)\prt_x p_t^0(x,y)+ \Big(L_{x}^{(\alpha),x,t}-\Lba_{x}^{(\alpha),y,t}\Big)p_t^0(x,y)
+L_{x}^{\nu,x,t}p_t^0(x,y)
\\&=:\Phi_t^{drift}(x,y)+\Phi_t^{(\alpha)}(x,y)+\Phi_t^\nu(x,y).
\ea
\ee

 In what follows, we estimate separately the components of $\Phi$ in the decomposition \eqref{Phi_decomp} and  deduce  an integral estimate for the entire $\Phi$, which holds true under $\mathbf{H}^{\nu}$ (i).  We will repeatedly use representation \eqref{p^01} and the following observation. The functions $\wt\lambda_t(z), \wt\rho_t(z), \wt\upsilon_t(z)$ are bounded  since they are obtained by averaging of bounded functions  w.r.t. probability measures. In addition, $\wt\lambda_t(z)$ is uniformly separated from zero. That is, for the function \eqref{wtg} with $z=y$ the bounds \eqref{g_bound0}, \eqref{g_bound_x} -- \eqref{g_bound_Lasym} can be used.

\emph{{Step 1: Estimate for $\Phi^{drift}$.}} By \eqref{b-B},\eqref{Lip_B} we have $$\ba
|b_t(x)-B_t(\kappa_t(y))|\leq |b_t(x)-B_t(x)|&+|B_t(x)-B_t(\kappa_t(y))|
\leq C\left(1+\left|{\kappa_t(y)-x\over t^{1/\alpha}}\right|\right)t^{-1+1/\alpha+\delta}.
\ea
$$
We have
$$
\prt_x p_t^0(x,y)=-{1\over t^{2/\alpha}}\Big(\widetilde{g}^{t,y}\Big)'\left({\kappa_t(y)-x\over t^{1/\alpha}}\right).
$$
Applying \eqref{g_bound_x}, and then  \eqref{G_comp}, \eqref{G_pol}, we easily get
\be\label{Phi_drift}\ba
|\Phi^{drift}_t(x,y)|&\leq Ct^{-1-1/\alpha+\delta}G^{(\alpha)}\left({\kappa_t(y)-x\over t^{1/\alpha}}\right)
=Ct^{-1+\delta}G^{(\alpha, \alpha,\alpha)}_t(x,\kappa_t(y)).
\ea
\ee

\emph{{Step 2: Estimate for $\Phi^{(\alpha)}$.}} Denote for $f\in C_\infty^2$
\be\label{Lsym}
L^{(\alpha),sym}f(x)=\int_{\Re}\Big(f(x+u)-f(x)-uf'(x)1_{|u|\leq 1}\Big){du\over |u|^{\alpha+1}},
\ee
\be\label{Lasym}
L^{(\alpha),asym}f(x)=\int_{\Re}\Big(f(x+u)-f(x)-uf'(x)1_{|u|\leq 1}\Big)\, \sgn(u){du\over |u|^{\alpha+1}}.
\ee
Then
$$
L^{(\alpha),z,t}f\left({x\over t^{1/\alpha}}\right)={\lambda(z)\over t}(L^{(\alpha),sym}f)\left({x\over t^{1/\alpha}}\right)+{\lambda(z)\rho(z)\over t}(L^{(\alpha),asym}f)\left({x\over t^{1/\alpha}}\right),
$$
$$
\widetilde{L}^{(\alpha),z,t}f\left({x\over t^{1/\alpha}}\right)={\lambda(\kappa_t(z))\over t}(L^{(\alpha),sym}f)\left({x\over t^{1/\alpha}}\right)+{\lambda(\kappa_t(z))\rho(\kappa_t(z))\over t}(L^{(\alpha),asym}f)\left({x\over t^{1/\alpha}}\right),
$$
and thus
\be\label{Phi_a}\ba
\Phi^{(\alpha)}_t(x,y)={\lambda(x)-\lambda(\kappa_t(y))\over t}&(L^{(\alpha),sym}\wt g^{t,y})\left({\kappa_t(y)-x\over t^{1/\alpha}}\right)
\\&+{\lambda(x)\rho(x)-\lambda(\kappa_t(y))\rho(\kappa_t(y))\over t}(L^{(\alpha),asym}\wt g^{t,y})\left({\kappa_t(y)-x\over t^{1/\alpha}}\right).
\ea
\ee
On the other hand, we have by \eqref{g_bound_Lsym}, \eqref{g_bound_Lasym}
$$
|L^{(\alpha),sym}\wt g^{t,z}(x)|+|L^{(\alpha),asym}\wt g^{t,z}(x)|\leq CG^{(\alpha)}(x).
$$
 Since  the functions  $\lambda(x)$ and $\rho(x)$ are bounded and $\zeta$-Holder continuous, this  gives
$$\ba
\Phi^{(\alpha)}_t(x,y)&\leq C\big(|x-\kappa_t(y)|^\zeta\wedge 1\big)t^{-1}G^{(\alpha, \alpha,\alpha)}_t(x, \kappa_t(y))=Ct^{-1+{\zeta/\alpha}}G^{(\alpha, \alpha-\zeta,\alpha)}_t(x,\kappa_t(y))
\\&=Ct^{-1+\delta_\zeta}G^{(\alpha, \alpha-\zeta,\alpha)}_t(x,\kappa_t(y)).
\ea
$$

\emph{Step 3: Estimate for $\Phi^{\nu}$.} We decompose
$$\ba
\Phi^{\nu}_t(x,y)=&\int_{\Re}\Big(p_t^0(x+u,y)-p_t^0(x,y)-u\prt_xp_t^0(x+u,y)1_{|u|\leq t^{1/\alpha}}\Big)\nu(x;du)
\\=&\int_{|u|\leq t^{1/\alpha}}\Big(p_t^0(x+u,y)-p_t^0(x,y)-u\prt_x p_t^0(x+u,y)\Big)\nu(x;du)
\\&+\int_{|u|> t^{1/\alpha}}p_t^0(x+u,y)\nu(x;du)
-\int_{|u|> t^{1/\alpha}}p_t^0(x,y)\nu(x;du)
\\&=:\Phi^{\nu,small}_t(x,y)+\Phi^{\nu, large,+}_t(x,y)+\Phi^{\nu, large,-}_t(x,y).
\ea
$$
  We have by \eqref{g_bound_x}
$$
|\prt^2_{xx} p_t^0(x,y)|\leq Ct^{-3/\alpha}G^{(\alpha+2)}\left({\kappa_t(y)-x\over t^{1/\alpha}}\right),
$$
which gives
$$\ba
|\Phi^{\nu,small}_t(x,y)| &\leq Ct^{-3/\alpha}\sup_{|v|\leq t^{1/\alpha}}G^{(\alpha+2)}\left({\kappa_t(y)-x-v\over t^{1/\alpha}}\right) \int_{|u|\leq t^{1/\alpha}}u^2|\nu|(x;du)
\\&\leq Ct^{-3/\alpha}G^{(\alpha+2)}\left({\kappa_t(y)-x-v\over t^{1/\alpha}}\right)(t^{1/\alpha})^{2-\beta},
\ea
$$
in the last inequality we used \eqref{vague}, condition \eqref{cond_ups_small}, and \eqref{est_2}. Next, we have by \eqref{g_bound0}
$$
p_t^0(x,y)\leq Ct^{-1/\alpha}G^{(\alpha)}\left({\kappa_t(y)-x\over t^{1/\alpha}}\right),
$$
thus by \eqref{cond_ups_small}
$$\ba
|\Phi^{\nu, large,-}_t(x,y)|\leq
p_t^0(x,y)\int_{|u|> t^{1/\alpha}}|\nu|(x;du)\leq Ct^{-1/\alpha}G^{(\alpha)}\left({\kappa_t(y)-x\over t^{1/\alpha}}\right)\Big(t^{1/\alpha}\Big)^{-\beta}.
\ea
$$ Then by \eqref{G_comp}
$$\ba
|\Phi^{\nu,small}_t(x,y)|+|\Phi^{\nu, large,-}_t(x,y)|&\leq Ct^{-3/\alpha}G^{(\alpha+2)}\left({\kappa_t(y)-x-v\over t^{1/\alpha}}\right)(t^{1/\alpha})^{2-\beta}
\\&+Ct^{-1/\alpha}G^{(\alpha)}\left({\kappa_t(y)-x\over t^{1/\alpha}}\right)\Big(t^{1/\alpha}\Big)^{-\beta}
\\&\hspace*{-4cm}\leq Ct^{-1/\alpha-\beta/\alpha}G^{(\alpha)}\left({\kappa_t(y)-x\over t^{1/\alpha}}\right)=Ct^{-\beta/\alpha}G^{(\alpha, \alpha,\alpha)}_t(x,\kappa_t(y))=Ct^{-1+\delta_\beta}G^{(\alpha, \alpha,\alpha)}_t(x,\kappa_t(y)).
\ea
$$
That is, the first and the third parts in the above decomposition of $\Phi^\nu$  satisfy a bound similar to the bound \eqref{Phi_drift} for $\Phi^{drift}$. For the second part,  we simply
write
$$
|\Phi^{\nu,large,+}_t(x,y)|\leq t^{-1+\delta_\beta}Q_t(x,y),
$$
where
\be\label{Q}
Q_t(x,y)=t^{\beta/\alpha}|\Phi^{\nu,large,+}_t(x,y)|=t^{1-\delta_\beta}|\Phi^{\nu,large,+}_t(x,y)|
\ee
is just a notation. This gives
$$
|\Phi^{\nu}_t(x,y)|\leq Ct^{-1+\delta_\beta}G^{(\alpha, \alpha,\alpha)}_t(x,\kappa_t(y))+Ct^{-1+\delta_\beta} Q_t(x,y).
$$

\emph{Summary: Proof of \eqref{Phi_bound}.}
The above calculation gives
\be\label{Phi_bound_final}
|\Phi_t(x,y)|\leq Ct^{-1+\delta}G^{(\alpha, \alpha,\alpha)}_t(x,\kappa_t(y))+C^{-1+\delta_\zeta}G^{(\alpha, \alpha-\zeta,\alpha)}_t(x,\kappa_t(y))+
Ct^{-1+\delta_\beta} Q_t(x,y),
\ee
and thus
\be\label{Phi_bound_vague}
|\Phi_t(x,y)|\leq  Ct^{-1+\delta}H_t(x,y), \quad H_t(x,y)=G^{(\alpha, \alpha-\zeta,\alpha)}_t(x,\kappa_t(y))+Q_t(x,y).
\ee
We have for any $\alpha, \beta, \gamma>0$
\be\label{G_kappa_bound}
\sup_x\int_{\Re}G^{(\alpha, \beta,\gamma)}_t(x,\kappa_t(y))\, dy\leq C,
\ee
see Appendix \ref{sA4}. Since
 $$
 \wt g^{t,y}(w)\leq C G^{(\alpha)}(w), \quad w\in \Re
 $$ by \eqref{g_bound0}, we have then
 \be\label{p_0_kernel}
 p_t^0(x,y)\leq C  G^{(\alpha, \alpha,\alpha)}_t(x+u,\kappa_t(y)).
\ee
Applying \eqref{G_kappa_bound} with $\alpha=\beta=\gamma$, we get by \eqref{cond_ups_small}
 \be\label{Q_bound}
 \ba
\sup_x\int_{\Re} Q_t(x,y)\, dy&\leq
 Ct^{-\beta/\alpha}\int_{|u|> t^{1/\alpha}}\left(\int_{\Re} G^{(\alpha, \alpha,\alpha)}_t(x+u,\kappa_t(y))\,dy\right)|\nu|(x;du)
 \\&\leq
 Ct^{\beta/\alpha}\int_{|u|> t^{1/\alpha}}\,|\nu|(x;du)\leq C.
 \ea
\ee
Applying once again \eqref{G_kappa_bound} with $\gamma=\alpha, \beta=\alpha-\zeta$, we get
\be\label{H_bound}
\sup_x\int_{\Re}H_t(x,y)\, dy\leq C,
\ee
which combined with \eqref{Phi_bound_vague} completes the proof of \eqref{Phi_bound}.

\begin{rem}\label{rem_H_tail} Using \eqref{cond_ups_large}, we can also get
\be\label{H_tails}
\quad \sup_{x\in \Re, t\in (0, T]}\int_{\{y:|y-x|>R\}}H_t(x,y)\, dy\to 0, \quad R\to \infty.
\ee
The proof is completely analogous  and is omitted.
\end{rem}

\subsection{Solution to \eqref{eq_int}: specification and further re-arrangement}\label{s44}

For any $k>1$ we have
$$
\Phi^{\star k}_t(x,y)=\int_{0<s_1<\dots<s_{k-1}<t}\Phi_{s_1-s_0,\dots, s_k-s_{k-1}}(x,y)\, ds_1\dots ds_{k-1},
$$
where we denote $s_0=0, s_k=t$,
$$
\Phi_{\tau_1,\dots, \tau_k}(x,y)=\Big(\Phi_{\tau_1}\ast\dots \ast \Phi_{\tau_k}\Big)(x,y)=\int_{\Re^{k-1}}\Phi_{\tau_1}(x,w_1)\dots \Phi_{\tau_k}(w_{k-1},y)\, dw_1 \dots dw_{k-1}.
$$
By \eqref{Phi_bound_vague},
$$\ba
\int_{\Re}&|\Phi_{\tau_1,\dots, \tau_k}(x,y)|\, dy\leq \int_{\Re^{k-1}}\int_\Re|\Phi_{\tau_1}(x,w_1)\dots \Phi_{\tau_k}(w_{k-1},y)|\, dw_1 \dots dw_{k-1}dy
\\&\leq C\tau_k^{-1+\delta}\int_{\Re^{k-1}}|\Phi_{\tau_1}(x,w_1)\dots \Phi_{\tau_{k-1}}(w_{k-2},w_{k-1})|\, dw_1 \dots dw_{k-1}\leq \dots\leq C^k\prod_{j=1}^k\tau_j^{-1+\delta}.
\ea
$$
Thus
$$\ba
\sup_x \int_{\Re} &|\Phi^{\star k}_t(x,y)|\, dy\leq C^k
\int_{0<s_1<\dots<s_{k-1}<t}\prod_{j=1}^k(s_j-s_{j-1})^{-1+\delta}\, ds_1\dots ds_{k-1}
\\&= C^kt^{-1+k\delta}
\int_{0<\upsilon_1<\dots<\upsilon_{k-1}<1}\prod_{j=1}^k(\upsilon_j-\upsilon_{j-1})^{-1+\delta}\, d\upsilon_1\dots d\upsilon_{k-1}=t^{-1+k\delta}{C^k\Gamma(\delta)^k\over \Gamma(k\delta)},
\ea
$$
which is just \eqref{Phi_bound_k}. That is, the solution $p_t(x,y)$ to the integral equation \eqref{eq_int} is uniquely defined by \eqref{sol_1}.

Note that the resolvent kernel $\Psi_t(x,y)$ for the integral equation \eqref{eq_int} inherits from $\Phi_t(x,y)$  the integral bounds and the tail behavior. Namely,  we have
\be\label{H_bound_Psi}
\sup_x \int_{\Re} |\Psi_t(x,y)|\, dy\leq \sum_{k=1}^\infty t^{-1+k\delta}{C^k\Gamma(\delta)^k\over \Gamma(k\delta)}\leq Ct^{-1+\delta}.
\ee
Next, by \eqref{H_tails} we have
$$
\sup_{x\in \Re, t\in (0, T]}t^{1-\delta}\int_{\{y:|y-x|>R\}}|\Phi_t(x,y)|\, dy\to 0, \quad R\to \infty.
$$
Then it is easy to show by induction that, for any $k$,
$$
 \sup_{x\in \Re, t\in (0, T]}t^{1-\delta}\int_{\{y:|y-x|>R\}}|\Phi_t^{\star k}(x,y)|\, dy\to 0, \quad R\to \infty.
$$
These bounds combined with \eqref{Phi_bound_k} yield the similar tail behavior of the kernel $\Psi_t(x,y)$:
\be\label{Psi_tail}
\sup_{x\in \Re, t\in (0, T]}t^{1-\delta}\int_{\{y:|y-x|>R\}}|\Psi_t(x,y)|\, dy\to 0, \quad R\to \infty.
\ee

The solution to \eqref{eq_int} can be written as
\be\label{sol}
p_t(x,y)=p_t^0(x,y)+r_t(x,y), \quad r_t(x,y)=(p^0\star\Psi)_t(x,y),
\ee
and by \eqref{H_bound_Psi},  \eqref{p_0_kernel}, and \eqref{G_kappa_bound} we have
\be\label{r_bound}
\ba
\sup_x\int_\Re |r_t(x,y)|\, dy&\leq \int_0^t\int_\Re\int_\Re p_{t-s}^0(x,z)|\Psi_s(z,y)|\, dz dy ds
\\&\leq C \int_0^t\left(\int_\Re p_{t-s}^0(x,z)\, dz \right) s^{-1+\delta} ds\leq Ct^\delta.
\ea
\ee
Note that representation \eqref{sol} differs from the one claimed in Theorem \ref{mainthm2}, in particular, the zero order term $p_t^0(x,y)$ in \eqref{sol} is not equal to the principal term
$$
p_t^{main}(x,y)={1\over t^{1/\alpha}}g^{t,x}\left({y-\mathfrak{f}_t(x)\over t^{1/\alpha}}\right)
$$
in \eqref{p_rep}. The difference between these two terms admits the following bound; the proof is postponed to Appendix \ref{sA41}:
\be\label{R_tilde_bound}
|p_t^{main}(x,y)- p_t^0(x,y)|\leq Ct^{\delta_\zeta}G_t^{(\alpha,\alpha-\zeta,\alpha)}(\chi_t(x), y)+Ct^\delta G_t^{(\alpha, \alpha, \alpha)}(\chi_t(x), y).
\ee
We have  \be\label{bint}
\sup_x\int_{\Re}G_t^{(\alpha,\beta,\gamma)}(x,y)\, dy\leq C,
\ee
see Appendix \ref{sA4}. That is, by \eqref{R_tilde_bound}
\be\label{R_tilde_bound_int}
\int_{\Re}|p_t^{main}(x,y)- p_t^0(x,y)|\, dy\leq C t^\delta.
\ee
Now it is easy to prove the following.

\begin{lem}\label{l_delta} For any $f\in C_\infty$,
$$
  \sup_{x}\left|\int_{\Re}p_t^{main}(x,y)f(y)\, dy-f(x)\right|\to 0,\quad \sup_{x}\left|\int_{\Re}p_t^0(x,y)f(y)\, dy-f(x)\right|\to 0, \quad t\to 0,
  $$
 \end{lem}
 \begin{proof}  We first note that there exists $C>1$ such that, for $|x|$ large enough,
$$
C^{-1}|x|\leq |\chi_t(x)|\leq  C|x|,
$$
see Proposition \ref{pCflow}. Since $f(x)\to 0, |x|\to \infty$, this gives
 \be\label{delta_f}
 \sup_x|f(\chi_t(x))-f(x)|\to 0, \quad t\to 0.
 \ee
 Next, $g^{t,x}$ are stable densities with uniformly bounded intensities and shifts, and thus   for every $\eps>0$
$$
\sup_{x}\int_{|w|>\eps}{1\over t^{1/\alpha}}g^{t,x}\left({w\over t^{1/\alpha}}\right)\, dw\to 0, \quad t\to 0.
 $$
Since $f\in C_\infty$ is uniformly continuous, this yields
$$
 \sup_{x}\left|\int_{\Re}{1\over t^{1/\alpha}}g^{t,x}\left({y-\chi_t(x)\over t^{1/\alpha}}\right)  f(y)  \, dy-f(\chi_t(x))\right|\to 0, \quad t\to 0,
 $$
 which proves the first assertion.  The second assertion follows from the first one by \eqref{R_tilde_bound_int}.
  \end{proof}

\section{Proof of Theorem \ref{mainthm1}}\label{s5}

We have defined the function $p_t(x,y)$ as a solution to the {integral} equation \eqref{eq_int}. In this section we make a further analysis of its representation \eqref{sol_1} and prove that function $p_t(x,y)$, in a  certain approximate sense, provides  a fundamental solution  to the Cauchy problem for the operator $\prt_t-L$. This fact will be a cornerstone for  the   proof of Theorem \ref{mainthm1}.

\subsection{Continuity properties and approximate fundamental solution}\label{s45}
Denote
$$
P_tf(x)=\int_{\Re}f(y)p_t(x,y)\, dy,\quad t>0, \quad P_0f(x)=f(x).
$$

\begin{lem}\label{l1}
For a given  bounded measurable $f$, the function $P_tf(x)$ is continuous w.r.t. $(t,x)\in (0, \infty)\times \Re$.

For $f\in C_\infty$, one has  $P_tf\in C_\infty, t\geq 0$, and $P_t, t\geq 0$ is a continuous family of bounded linear operators in $C_\infty$.
\end{lem}

\begin{proof}
The proof is fairly standard, thus  we just sketch it.  We have
 \be\label{sol_2}
 P_tf(x)=\int_{\Re}f(y)p_t^0(x,y)\, dy+\int_0^t\int_{\Re}p^0_{t-s}(x,y)\Psi_s^f(y)\, dy ds,
 \ee
\be\label{Psi_f}
 \Psi_t^f(y)=\int_{\Re}\Psi_t(y,z)f(z)\, dz.
 \ee
The function  $p_t^0(x,y)$, given by an explicit formula  \eqref{p^0}, is continuous w.r.t. $x,t$ for any $y$. Then one can deduce
continuity of $P_tf(x)$ using the bounds \eqref{p_0_kernel}, \eqref{H_bound_Psi} and a standard domination convergence argument; e.g. \cite[Section 3.3]{KK15}. Using \eqref{Psi_tail}, one can show in addition that
\be\label{conv_x}
  P_{t}f(x)\to 0, \quad |x|\to \infty
  \ee
  uniformly in $t\in [0,T ].$ Combined with continuity of $P_tf(x)$ in $(t,x)\in (0, T]\times \Re$ and Lemma \ref{l_delta},
  this yields continuity in $t\in [0,T]$ of the family $\{P_tf\}\in C_\infty.$ Clearly, each $P_t$ is a linear operator; these operators are bounded thanks to \eqref{p_0_kernel}, \eqref{H_bound_Psi}.
  \end{proof}

 \begin{lem}\label{l2}
For a given  $f\in C_\infty$, the function
$\Psi^f_t(x)$ is continuous w.r.t. $(t,x)\in (0, \infty)\times \Re$.
In addition, for any $0<\tau<T$
\be\label{111}
\Psi^f_t(x)\to 0,\quad  |x|\to \infty
\ee
uniformly on $t\in [\tau, T]$.
\end{lem}
  \begin{proof} The argument here is close to the one from the previous proof, with $p_t^0$ changed to $\Phi$; recall that $\Psi_t(x,y)$ satisfies
   \be\label{Psi_2}
 \Psi_t(x,y)=\Phi_t(x,y)+\int_0^t\int_{\Re}\Phi_{t-s}(x,y')\Psi_s(y',y)\,dy.
 \ee
Therefore we omit the details, and only discuss two points which make the difference with the previous proof. First, the bound \eqref{Phi_bound_vague}, when compared to \eqref{p_0_kernel}, contains an extra term $t^{-1+\delta}$. This is the reason why \eqref{111} is stated for $t\in [\tau, T]$ with \emph{positive}  $\tau$. Next, we yet have to verify that $\Phi_t(x,y)$ is continuous in $x,t$. Recall the decomposition \eqref{Phi_decomp}, and observe that the term $\Phi^{(\alpha)}$ has the required continuity. However, two other terms in the decomposition \eqref{Phi_decomp} may fail to be continuous. Namely, since the function $1_{|u|>t^{1/\alpha}}$ is discontinuous,  weak continuity of the kernel $\nu(x, du)$ does not imply, in general, continuity of the corresponding  integral $m_t^\nu(x)$. This trouble is artificial, and can be fixed by a proper re-arrangement of the compensating terms in these two summands. Namely, we
  take function $\theta\in C(\Re)$ with
 $$
 \theta(u)=0, \quad |u|\leq {1\over 2}, \quad \theta(u)=1, \quad |u|\geq 1,
 $$
 and put
 $$
 \widehat{m}_t^\nu(x)=\int_{|u|\leq 1}u\theta(ut^{-1/\alpha})\nu(x,du), \quad \widehat b_t=b-m_t^{(\alpha)}-\widehat m_t^{\nu},
 $$
 $$
 \widehat{\Phi}^{drift}_t(x,y)=\Big(\widehat{b}_t(x)-B_t(\kappa_t(y))\Big)\prt_x p_t^0(x,y).
 $$
 $$
 \widehat{\Phi}^{\nu}_t(x,y)=\int_{\Re}\Big(p_t^0(x+u,y)-p_t^0(x,y)-u\big(1-\theta(ut^{-1/\alpha})\big)\prt_x p_t^0(x,y)1_{|u|\leq 1}\Big)\nu(x;du).
$$
Then
$$
\widehat{\Phi}^{drift}_t(x,y)+ \widehat{\Phi}^{\nu}_t(x,y)={\Phi}^{drift}_t(x,y)+ {\Phi}^{\nu}_t(x,y),
$$
and the terms $\widehat{\Phi}^{drift}$ and $\widehat{\Phi}^{\nu}$ have the required continuity. The latter can be verified via  a routine calculation  involving the continuity condition $\mathbf{H}^{cont}$, we omit a detailed discussion.
  \end{proof}

The parametrix construction described in Section \ref{s41} originates in the general interpretation of $p_t(x,y)$ as a (sort of) fundamental solution to the Cauchy problem for the operator $\prt_t-L$; that is, in other words, $p_t(x,y)$ should satisfy  the backward Kolmogorov equation for the (yet unknown) process $X$.  In some cases one can show that $p_t(x,y)$ indeed satisfies
\be\label{L_fund}
(\prt_t-L_x)p_t(x,y)=0
\ee
in a classical way; for instance, this is the mainstream approach in the  classical diffusive/parabolic setting,  see \cite{Fr64}. A necessary pre-requisite for such an approach is to prove that $p_t(x,y)$ belongs to $C^1$ w.r.t. $t$ and to $C^2_\infty$ (which is just the domain of $L$) w.r.t. $x$. In the current setting, zero order approximation $p_t^0(x,y)$ has the required smoothness properties, however one can hardly extend these properties to $p_t(x,y)$ using \eqref{sol_2} in the way used in the proof of Lemma \ref{l1}. The main obstacle is that $\prt_{x} p_t^0(x,y),\prt_{xx}^2 p_t^0(x,y)$  exhibit  strongly singular behavior as $t\to 0$ (see \eqref{g_bound_x}, \eqref{g_bound_xx}), which does not allow one to differentiate  \eqref{sol_2}. This observation leads to the following auxiliary construction.
Define for $\eps>0$
\begin{equation}\label{pe}
p_{t,\epsilon}(x,y)=p_{t+\epsilon}^0(x,y)  + \int_0^t \int_{\Re}  p_{t-s+\eps}^0(x,y')  \Psi_s(y',y) dy'ds,
\end{equation}
\begin{equation} \label{Pte}
P_{t,\epsilon} f(x)=\int_\rd p_{t,\epsilon}(x,y)f(y)dy, \quad f\in C_\infty.
\end{equation}
The following lemma shows that  $p_{t,\eps}(x,y)$ approximates $p_{t}(x,y)$ and satisfies an approximative analogue of \eqref{L_fund}. This is our reason to call the family   $\{p_{t,\eps}(x,y), \eps>0\}$ an \emph{approximate fundamental solution}.

\begin{lem}\label{l3} For every $f\in C_\infty$ we have the following.
\begin{enumerate}
 \item    For every  $T>0$,
  \be\label{conv_pte}
 \|P_{t,\epsilon}f- P_t f\|_\infty\to 0, \quad \epsilon\to 0,
  \ee
   uniformly in $t\in (0,T ]$, and
   \be\label{conv_x_eps}
  P_{t,\epsilon}f(x)\to 0, \quad |x|\to \infty
  \ee
  uniformly in $t\in (0,T ],\eps\in (0,1].$

  \item
     \be\label{L_delta_eps}
 \lim_{t,\eps\to 0+} \|P_{t,\eps} f-f\|_\infty =0.
  \ee

\item  For every  $\eps>0$,  $P_{t,\eps}f(x)$ belongs to $C^1$ as a function of $t$, to $C^2_\infty$ as a function of $x$, and  $\prt_tP_{t,\eps}f(x), L_x P_{t,\eps}f(x)$ are continuous w.r.t. $(t,x)$.
   \item For every $0<\tau<T$, $R>0$
  \be\label{conv_loc}
Q_{t,\eps} f(x)= \big(\partial_t-L_x\big) P_{t,\epsilon} f (x)\to 0, \quad \epsilon\to 0,
\ee
uniformly in  $(t,x)\in [\tau,T]\times [-R,R]$. In addition,
 \be\label{conv_int}
\int_0^T\sup_{x\in [-R,R]}|Q_{t,\eps} f(x)|\, dt\to 0, \quad \epsilon\to 0.
\ee

\end{enumerate}
\end{lem}
\begin{proof} Statements 1 -- 3 follow easily by the same continuity/domination argument which was used in  Lemma \ref{l1} and thus we omit the proof; see \cite[Section 4.1]{KK15} for a detailed exposition of similar group of statements.

To prove statement 4, we apply the argument from the proof of \cite[Lemma 5.2]{KK15}. Since the additional time shift by $\eps>0$ removes the singularity at the point $t=0$ in \eqref{pe}, the continuity/domination argument similar to the one used in Lemma \ref{l1}   allows one to interchange the operator $\big(\partial_t-L_x\big)$ with  the integrals in the definition of $P_{t,\eps}f$. Then, recalling the definition \eqref{Phi} of $\Phi_t(x,y)$  and  \eqref{Psi_f}, we get
$$
Q_{t,\eps} f(x)=-\int_{\Re}\Phi_{t+\epsilon}(x,y)f(y)\, dy-\int_0^t \int_{\Re} \Phi_{t-s+\eps}(x,y)  \Psi_s^f(y)\, dyds+
\int_{\Re} p_{\epsilon}^0(x,y) \Psi_t^f(y)\, dy,
$$
see \cite[(4.13)]{KK15}
By the continuity of $\Phi_{t}(x,y)$ in $t$, we have
$$\ba
\int_{\Re}\Phi_{t+\epsilon}(x,y)f(y)\, dy+\int_0^t \int_{\Re} \Phi_{t-s}(x,y)  \Psi_s^f(y)\, dyds&\to \int_{\Re}\Phi_{t}(x,y)f(y)\, dy
\\&+\int_0^t \int_{\Re} \Phi_{t-s+\eps}(x,y)  \Psi_s^f(y)\, dyds, \quad \eps\to 0
\ea
$$
uniformly in $x\in [-R,R], t\in [\tau, T]$. On the other hand, since $\Psi^f_t(x)$ is continuous, we have by Lemma \ref{l_delta}
$$
\int_{\Re} p_{\epsilon}^0(x,y) \Psi_t^f(y)\, dy\to \Psi_t^f(x), \quad \eps\to 0
$$
uniformly in $x\in [-R,R], t\in [\tau, T]$, which combined with \eqref{Psi_2}
completes the proof of \eqref{conv_loc}. On the other hand it follows from \eqref{p_0_kernel} and \eqref{H_bound_Psi} that
$$
\int_0^\tau\sup_x|Q_{t,\eps} f(x)|\, dt\leq C\tau^{\delta}\sup_y|f(y)|.
$$
Combined with \eqref{conv_loc}, this yields \eqref{conv_int}.
\end{proof}

\begin{dfn}\label{dah} We say a continuous function
$h(t,x)$ to be  \emph{approximate harmonic} for an operator $\prt_t-L$, if there exists a family $\{h_\eps(t,x), \eps\in (0,1]\}\in C([0,\infty)\times \Re)$ such that
\begin{itemize}
  \item[(i)]  for any $T>0$
  $$
  \sup_{x\in \Re,t\in [0,T]}|h_\eps(t,x)-h(t,x)|\to 0, \quad \eps\to 0, \quad \sup_{t\in [0,T], \eps\in (0,1]}|h_\eps(t,x)|\to 0, \quad |x|\to \infty;
  $$
  \item[(ii)] each function $h_\eps(t,x)$ is $C^1$ w.r.t. $t$, $C^2_\infty$ w.r.t. $x$, and for every $0<\tau<, R>0$
  $$
  \sup_{x\in [-R,R],t\in [\tau, T]}|(\prt_t-L_x)h_\eps(t,x)|\to 0, \quad \eps\to 0.
  $$
\end{itemize}
\end{dfn}

Note that, by Lemma \ref{l3}, for any  $f\in C_\infty$ the function
$h^f(t,x)=P_tf(x)$ is approximate harmonic for $\prt_t-L$.  The corresponding approximating family is given by
\be\label{hfs}
h^f_\eps(t,x)=P_{t,\eps}f(x), \quad \eps>0.
\ee

\subsection{The Positive Maximum Principle and the semigroup properties}\label{s46}

In this section we establish the  semigroup properties for the family  of the operators $\{P_t, t\geq0\}$.
A classical method for this is based on the \emph{Positive Maximum Principle} (PMP) for the operator $L$. It is usually applied when $p_t(x,y)$ is a (true) fundamental solution for $\prt_t-L$; e.g. \cite{Ko89}. In our setting  $p_t(x,y)$ satisfies \eqref{L_fund} in a weaker approximate sense; however, the classical PMP-based argument admits an extension which is well applicable in such an approximate setting. This extended argument is essentially due to  \cite[Section 4]{KK15}. For the reader's  and further reference convenience, here we give a systematic version of this argument, based on the  notion of  approximate harmonic functions.

Recall that an operator $L$ with a domain $\mathcal{D}$  is said to satisfy PMP if for any $f\in \mathcal{D}$ and $x_0$ such that
$$
0\leq f(x_0)=\max_x f(x)
$$
one has
$$
Lf(x_0)\leq 0.
$$
Clearly, the operator \eqref{L_gen1} with the domain $\mathcal{D}=C^2_\infty$ satisfies PMP; note that  $Lf$ is continuous for any $f\in C_\infty^2$, but does not necessarily  belong to $C_\infty$.

\begin{prop}\label{p_PMP} Let $h(t,x)$ be an approximate harmonic function for $\prt_t-L$ and $h(0, \cdot)\geq 0$.

Then $h(t,\cdot)\geq 0, t> 0$.
\end{prop}
\begin{proof} Assuming $h(t,x)$ being negative at some point, we have that for some $T>0$
\be\label{ass}
\inf_{t\leq T,x\in \Re} h(t,x)<0.
\ee
Let $\{h_\eps(t,x), \eps\in (0,1]\}$ be the approximating family from  Definition \ref{dah}, then
by  assertion (i)  there exist $\upsilon>0, \theta>0, \eps_1>0$ such that
$$
\inf_{t\leq T,x\in \Re} \Big(h_\eps(t,x)+\theta t\Big)<-\upsilon, \quad \eps<\eps_1.
$$
Denote
$$
u_{\epsilon}(t,x)= h_\eps(t,x)+\theta t,\quad \eps\in (0,1]
$$
these functions are continuous in $(t,x)$ (because each $h_\eps$ is continuous) and satisfy $$
u_{\epsilon}(t,x)\to \theta t>0, \quad |x|\to\infty
$$
uniformly in $t\in [0, T]$ (because of the assertion (i)). Then for some $R>0$ and $\eps<\eps_1$
$$
\inf_{t\leq T,x\in \Re} u_\eps(t,x)=\inf_{t\leq T,x\in \Re} \Big(h_\eps(t,x)+\theta t\Big)<0
$$
is actually attained at some point in $[0,T]\times [-R,R]$;  we fix one such a point for each $\eps$, and denote it by  $(t_\eps, x_\eps)$.
We observe that $t_\eps$ is separated from $0$ when $\eps$ is small enough. Indeed, by the assertion (i) and non-negativity assumption $h(0, x)\geq 0$, there exist $\eps_0>0$, $\tau>0$ such that
$$
u_{\epsilon}(t,x)\geq h_{\epsilon}(t,x)\geq  h_{\epsilon}(0,x)-{\upsilon\over 2}\geq -{\upsilon\over 2}, \quad t\leq \tau, \quad \eps<\eps_0, \quad x\in \Re.
$$
 Since
$$
u_{\epsilon}(t_\eps,x_\eps)=\min_{t\in [0, T], x\in \Re}u_\eps(t,x)<-\upsilon,
$$
this yields  $t_\eps>\tau$ for $\eps<\eps_0$.

Now we can conclude the proof in a quite standard way.  Let $\eps<\eps_0\wedge \eps_1$. Since $x_\eps$ is the maximal point for $-u_\eps(t_\eps, \cdot)$ and $-u_\eps(t_\eps, x_\eps)>0$, we have by the PMP
$$
L_x u_\eps(t_\eps,x_\eps)\geq 0.
$$
Since $t_\eps$ is the maximal point for $u_\eps(\cdot, x_\eps)$ and $t_\eps>\tau$,
we have
$$
\prt_t u_\eps(t_\eps,x_\eps)\leq 0,
$$
 where the  sign `$<$'  may appear only  if $t_\eps=T$. Then
\be\label{PMP}
(\prt_t-L_x) u_\eps(t,x)|_{(t,x)=(t_\eps,x_\eps)}\leq 0.
\ee
On the other hand, we  by the assertion (ii) from  Definition \ref{dah}
$$
(\prt_t-L_x) u_\eps(t,x)|_{(t,x)=(t_\eps,x_\eps)}\geq \theta + \inf_{x\in [-R,R], t\in [\tau, T]}(\prt_t-L_x) h_\eps(t,x) \to \theta>0, \quad \eps\to 0.
$$
This gives contradiction and shows that \eqref{ass} fails.
\end{proof}

Now the semigroup properties for the family $\{P_t, t\geq0\}$ can be derived in a standard way.

\begin{cor} \begin{enumerate}
              \item Each operator $P_t, t\geq 0$ is positivity preserving: for any $f\geq 0$ one has  $P_t f \geq 0 $.
              \item The family $\{P_t\}$ is a semigroup: \be\label{semigr}
P_{t+s}f=P_tP_sf, \quad f\in C_\infty, \quad s,t\geq 0.
\ee
              \item For any $f\in C_0^2(\Re)$,
             \begin{equation}\label{Dy1}
P_tf(x)-f(x)=\int_0^tP_sLf(x)\, ds,  \quad t\geq 0.
\end{equation}
            \end{enumerate}
\end{cor}

\begin{proof} Statement 1 follows from Proposition \ref{p_PMP} applied to $h(t,x)=h^f(t,x)$, which is already known to be approximate harmonic.  To prove statement 2, we fix $s\geq 0,$ $f\in C_\infty$ and apply  Proposition \ref{p_PMP} to functions
$$
h^{\pm}(t,x)=\pm
P_{t+s}f(x)\mp P_tP_sf(x)=\pm h^f(t,x)\mp h^{P_sf}(t,x),
$$
which are  approximate harmonic and satisfy $h^{\pm}(0, \cdot)=0$. Finally, to prove statement 3 we apply Proposition \ref{p_PMP}
to the function
$$
h(t,x)=P_tf(x)-f(x)-\int_0^tP_sLf(x)\, ds,
$$
with the approximating family defined by  $$
h_\eps(t,x)=P_{t, \eps}f(x)-f(x)-\int_0^tP_{s,\eps}Lf(x)\, ds.
$$
Note that  $h_\eps(t,x)$ satisfies assertion (i) from Definition \ref{dah}  by Lemma \ref{l3}, and
 $$\ba
 (\prt_t-L_x)h_\eps(t,x)&=Q_{t,\eps}f(x)-\Big(0-Lf(x)\Big)-\left(P_{t, \eps}Lf(x)-\int_0^tL\Big(P_{s, \eps} Lf\Big)(x)\, ds\right)
 \\&=Q_{t,\eps}f(x)-\int_0^t\prt_s\Big(P_{s, \eps} Lf\Big)(x)+\int_0^tL\Big(P_{s, \eps} Lf\Big)(x)\, ds
 \\&=Q_{t,\eps}f(x)-\int_0^tQ_{s,\eps}Lf(x)\, ds.
 \ea
 $$
Applying \eqref{conv_loc} and \eqref{conv_int}, we get  assertion (ii) from Definition \ref{dah}.
\end{proof}

It is easy to deduce from \eqref{Dy1} that
$$
\int_\Re p_t(x,y)\, dy=1, \quad t>0, \quad x\in \Re.
$$
Indeed, take $f\in C^0_\infty$ such that $f(x)=1, |x|\leq 1$, and put $f_k(x)=f(k^{-1}x)$. Then
$$
f_k(x)\to 1, \quad Lf_k(x)\to 0, \quad k\to \infty
$$
for every $x$, and $\|Lf_k\|\leq C$. Using \eqref{sol}, \eqref{p_0_kernel},  and \eqref{r_bound} we can apply the dominated convergence theorem and prove
$$
\int_0^tP_sLf_k(x)\, ds\to 0, \quad k\to \infty,
$$
which combined with \eqref{Dy1} gives the required identity.

Summarizing all the above, we conclude that $P_t, t\geq 0$ is  a \emph{strongly continuous semigroup} in $C_\infty$, which is \emph{positivity preserving} and \emph{conservative}; that is, this semigroup is \emph{Feller}. It follows from \eqref{Dy1} that
$C_0^2$ belongs to the  domain of its generator, and the restriction of this generator to $C_0^2$ equals $L$.
 For any probability measure $\pi$ on $\Re$ there exists a Markov process $\{X_t\}$ with the transition semigroup $\{P_t\}$, c\`adl\`ag trajectories, and the initial distribution $\mathrm{Law}\,(X_0)=\pi$; see  \cite[Theorem 4.2.7]{EK86}.  Finally, by Lemma \ref{l1}  the process $X$ is strong Feller.

\subsection{The martingale problem: uniqueness}\label{s47}
Note that any Feller process $Y$, whose generator $A$ restricted to $C_0^2$ coincides with $L$, is a $D(\Re^+)$-solution to the martingale problem $(L, C_0^2)$; this is essentially the \emph{Dynkin formula} combined with \cite[Theorem 4.2.7]{EK86}. In particular, this is the case for the Markov process $X$, constructed in the previous section. In this section, we prove that the  $D(\ax)$-solution to the martingale problem $(L, C_0^2)$ with a given initial distribution $\pi$ is unique; this will complete the proof of Theorem \ref{mainthm1}.  The argument here is principally the same as in \cite{Ku18}, with the one important addition which appears because the drift term now is not necessarily bounded.

By  \cite[Corollary 4.4.3]{EK86}, the required uniqueness holds true if for any two $D(\ax)$-solutions to $(L, C_0^2)$ with the same initial distribution $\pi$ corresponding one-dimensional distributions coincide. In what follows, we fix \emph{some} solution $Y$ and prove that
\be\label{onedim}
\E f(Y_T)=\int_{\Re}P_Tf(x)\pi(dx), \quad f\in C_\infty, \quad T>0.
\ee

It is easy to prove that $Y_t, t\geq 0$ is stochastically continuous; see \cite{KK15}. Then for any function $h(t, x)$ which is differentiable w.r.t. $t$, belongs to $C^2_0$ w.r.t. $x$, and has continuous and bounded $\prt_t h(t, x), L_xh(t,x)$,  the process
$$
h(t, Y_t)-\int_0^t\Big( \prt_sh(s, Y_s)+L_xh(s, Y_s)\Big)\, ds
$$
is a martingale,  see  \cite[Lemma 4.3.4~(a)]{EK86}. We use this fact for a certain family of functions which approximate
$$
h^{T,f}(t,x)=P_{T-t}f(x), \quad t\in [0,T],\quad  x\in \Re;
$$
here and below $f\in C_\infty, T>0$ are fixed. Consider a family of functions $\{\phi_R, R>0\}\subset C^2$ such that $\|\phi_R\|_{C^2}\leq C$ and
$$
\phi_R(x)=\left\{
            \begin{array}{ll}
              1, & |x|\leq R+1; \\
              0, & |x|\geq R+2.
            \end{array}
          \right.
$$
Define
$$
h^{T,f}_{R, \eps}(t,x)=\phi_R(x)P_{T-t, \eps}f(x),\quad R>0, \quad  \eps>0.
$$
Recall that $P_{T-t, \eps}f(x)\in C^2,$ and is bounded together with its derivatives uniformly for $t\in [0, T_1], |x|\leq R$ for any $T_1<T, R>0$.  Multiplying this function by $\phi_R$, we get a function from the class $C_0^2$. That is, we have that
$$
M^{T,f}_{R, \eps}(t)=h^{T,f}_{R, \eps}(t, Y_t)-\int_0^t\Big( \prt_sh^{T,f}_{R, \eps}(s, Y_s)+L_xh^{T,f}_{R, \eps}(s, Y_s)\Big)\, ds, \quad t\in [0, T_1]
$$
is a martingale. Denote  $h^{T,f}_{\eps}(t,x)=P_{T-t, \eps}f(x)$. It is clear that
$$
\prt_th^{T,f}_{R, \eps}(t, x)=\phi_R(x)h^{T,f}_{\eps}(t,x).
$$
In addition, we have
$$
L_x h^{T,f}_{R, \eps}(t, x)=\phi_R(x)L_xh^{T,f}_{\eps}(t,x)+\int_{|u|\geq 1}\Big(\phi_R(x+u)-\phi_R(x)\Big)h^{T,f}_{\eps}(t,x+u)\mu(x,du), \quad |x|\leq R.
$$
Thus for $|x|\leq R$ we can write
$$
\prt_th^{T,f}_{R, \eps}(t, x)+L_xh^{T,f}_{R, \eps}(t, x)=-\phi_R(x)Q_{T-s, \eps}f(x)+\Theta^{T,f}_{R, \eps}(t,x),
$$
where $Q_{t,\eps}f$ is defined in  Lemma \ref{l3}, and
$$
\Theta^{T,f}_{R, \eps}(t,x)=\int_{|u|\geq 1}\Big(\phi_R(x+u)-\phi_R(x)\Big)h^{T,f}_{\eps}(t,x+u)\mu(x,du)
$$
Observe that, for $|x|\leq R$,
$$
\Big(\phi_R(x+u)-\phi_R(x)\Big)\not=0\Rightarrow|x+u|\geq R+1,
$$
which yields
$$
|\Theta^{T,f}_{R, \eps}(t,x)|\leq C\sup_{t\in [0, T], |y|\geq R+1}| h^{T,f}_{\eps}(t,y)|=:F^{T,f}_{R, \eps}.
$$

Now we can finalize the proof. Without loss of generality, we assume that the initial distribution $\pi$ has a compact support, and take   $R$ large enough, so that $\mathrm{supp}\,\pi\subset(-R,R)$. Denote
$$
\tau^R=\inf\{t:|Y_t|\geq R\}>0,
$$
then for any $T_1<T$ we have
$$\ba
|\E h^{T,f}_{R,\eps}(T_1\wedge\tau_R,Y_{T_1\wedge\tau_R})&-\E h^{T,f}_{R,\eps}(0,Y_0)|
\\&\leq \E\int_0^{T_1\wedge\tau_R}|Q_{T-s, \eps}f(Y_s)|\, ds+TF^{T,f}_{R, \eps}.
\ea
$$
Using Lemma \ref{l3}, we pass to the limit as $\eps\to 0$ and get
$$
|\E h^{T,f}(T_1\wedge\tau_R,Y_{T_1\wedge\tau_R})-\E h^{T,f}(0,Y_0)|
\leq CT\sup_{t\in [0, T], |y|\geq R+1}| h^{T,f}(t,y)|.
$$
Taking $R\to \infty$ and using Lemma \ref{l1}, we get by the domination convergence theorem
$$
E h^{T,f}(T_1,Y_{T_1})=\E h^{T,f}(0,Y_0), \quad T_1<T.
$$
Taking $T_1\to T$ and using the domination convergence theorem again, we get
$$
\E f(Y_T)=\E P_Tf(Y_0)=\int_{\Re}P_Tf(x)\pi(dx),
$$
which proves \eqref{onedim}.

\section{Proof of Theorem \ref{mainthm2}}\label{s6}

\emph{Statement I} follows straightforwardly from \eqref{r_bound} and \eqref{R_tilde_bound}. To prove \emph{statement II,} we further re-arrange decomposition \eqref{Phi_decomp}. Namely, we write
\be\label{Phi_decomp2}
\Phi_t(x,y)=\Phi_t^{kernel}(x,y)+\Phi_t^{integral}(x,y),
\ee
where
$$
\Phi_t^{integral}(x,y)=\Phi_t^{\nu, large,+}=\int_{|u|> t^{1/\alpha}}p_t^0(x+u,y)\nu(x;du),
$$
and the term $\Phi_t^{kernel}(x,y)=\Phi_t(x,y)-\Phi_t^{integral}(x,y)$ admits a point-wise (`kernel') bound
$$
|\Phi_t^{kernel}(x,y)|\leq Ct^{-1+\delta} G^{(\alpha, \alpha-\zeta,\alpha)}_t(x,\kappa_t(y)).
$$
Since the kernel  $G^{(\alpha, \alpha-\zeta,\alpha)}_t(u,v)$ is bounded by $Ct^{-1/\alpha}$, satisfies \eqref{bint}, and is symmetric, one has
\be\label{kernel}
|\Phi_t^{kernel}(x,y)|\leq Ct^{-1/\alpha-1+\delta}, \quad \sup_y\int_{\Re} |\Phi_t^{kernel}(x,y)|\, dx\leq C t^{-1+\delta}.
\ee
Next, it is straightforward to see that $\Phi_t^{integral}(x,y)$ satisfies the similar $\sup$-bound: since $p_t^0(x,y)$  is bounded by $Ct^{-1/\alpha}$, we have by \eqref{cond_ups_small}, \eqref{cond_ups_large}
\be\label{integral}
|\Phi_t^{integral}(x,y)|\leq C t^{-1/\alpha}\int_{|u|> t^{1/\alpha}}|\nu|(x;du)\leq Ct^{-1/\alpha-1+\delta_\beta}.
\ee
To obtain an integral bound for $\Phi_t^{integral}(x,y)$, we recall that
$$
p_t^0(x+u,y)\leq {C\over t^{1/\alpha}}G^{(\alpha)}\left({\kappa_t(y)-x-u\over t^{1/\alpha}}\right),
$$
and observe that
$$
G^{(\alpha)}(x)\leq C\left(\1_{[-1,1]}+ G^{(\alpha)}\ast\1_{[-1,1]}\right)(x).
$$
Then by \eqref{invert}
$$\ba
\int_{\Re}& |\Phi_t^{integral}(x,y)|\, dx\leq C t^{-1/\alpha}
\left|\int_\Re \int_{|u|>t^{1/\alpha}} G^{(\alpha)}\left({w-x-u\over t^{1/\alpha}}\right)\nu(x; du)\, dx\right|
\\&
\leq C t^{-1/\alpha}\left|\int_\Re \int_{|u|>t^{1/\alpha}} \1_{[-1,1]}\left({w-x-u\over t^{1/\alpha}}\right)\nu(x; du)\, dx\right|
\\&+C t^{-1/\alpha}\left|\int_\Re \left(\int_\Re \int_{|u|>t^{1/\alpha}} \1_{[-1,1]}\left({w-x-u\over t^{1/\alpha}}-z\right) G^{(\alpha)}(z)\nu(x; du)\, dx\right)\, dz\right|\leq Ct^{-1+\delta_\nu}.
\ea
$$
Combined with \eqref{kernel}, \eqref{integral}, this yields
\be\label{full}
|\Phi_t(x,y)|\leq Ct^{-1/\alpha-1+\delta}, \quad \sup_y\int_{\Re} |\Phi_t(x,y)|\, dx\leq C t^{-1+ \delta_\infty}.
\ee
These bounds can be extended to the kernel $\Psi=\sum_{k\geq 1}\Phi^{\star k}$:
\be\label{full_Psi}
|\Psi_t(x,y)|\leq Ct^{-1/\alpha-1+\delta}, \quad \sup_y\int_{\Re} |\Psi_t(x,y)|\, dx\leq C t^{-1+ \delta_\infty}.
\ee
The second bound follows from the second bound in \eqref{full} literally in the same way with \eqref{H_bound_Psi}. To get the first bound, we slightly modify the argument from Section \ref{s44}. In what follows we use the notation of this section. Let $k\geq 1, \tau_1, \dots, \tau_k\in [0, T]$ be given, and let $j\in \{1, \dots,k\}$ be such that $\tau_j=\max_{i=1, \dots, k}\tau_i$.   Using the first inequality in \eqref{full} with $t=\tau_j$, we get
$$\ba
|\Phi_{\tau_1,\dots, \tau_k}&(x,y)|\leq \int_{\Re^{k-1}}\int_\Re|\Phi_{\tau_1}(x,w_1)\dots \Phi_{\tau_k}(w_{k-1},y)|\, dw_1 \dots dw_{k-1}
\\&\leq C\tau_j^{-1-1/\alpha+\delta}\int_{\Re^{j-1}}|\Phi_{\tau_1}(x,w_1)\dots\Phi_{\tau_{j-1}}(w_{j-2},w_{j-1})| \, dw_1\dots dw_{j-1}
\\&\hspace*{1cm}\times \int_{\Re^{k-j-1}}|\Phi_{\tau_{j+1}}(w_{j},w_{j+1})\dots \Phi_{\tau_{k-1}}(w_{k-2},w_{k-1})|\, dw_{j+1} \dots dw_{k-1}.
\ea
$$
Then, using repeatedly \eqref{Phi_bound} and the second inequality in \eqref{full} we get
$$\ba
|\Phi_{\tau_1,\dots, \tau_k}&(x,y)|\leq \\&\leq C\tau_j^{-1-1/\alpha+\delta}(C\tau_{j-1}^{-1+\delta})\int_{\Re^{j-2}}|\Phi_{\tau_1}(x,w_1)\dots\Phi_{\tau_{j-2}}(w_{j-3},w_{j-2})| \, dw_1\dots dw_{j-2}
\\&\hspace*{1cm}\times (C\tau_{j+1}^{-1+\delta_\infty})\int_{\Re^{k-j-2}}|\Phi_{\tau_{j+2}}(w_{j+1},w_{j+2})\dots \Phi_{\tau_{k-1}}(w_{k-2},w_{k-1})|\, dw_{j+2} \dots dw_{k-1}
\\&\leq \dots\leq C^k\Big(\prod_{i=1}^{j-1}\tau_i^{-1+\delta}\Big)\tau_j^{-1-1/\alpha+\delta}\Big(\prod_{i=j+1}^{k}\tau_i^{-1+\delta_\infty}\Big)\leq \tau_j^{-1-1/\alpha-\delta}\Big(\prod_{i\not=j, i\leq k}\tau_i^{-1+\delta_\infty}\Big).
\ea
$$
Now we take $0\leq s_1\leq \dots \leq s_{k-1}\leq t$ and put $s_0=0, s_k=t$, $\tau_i=s_i-s_{i-1}, i=1, \dots, k$. Then the maximal value $\tau_j$ is $\geq t/k$, and we get
$$\ba
|\Phi_{t}^{\star k}(x,y)|&\leq k^{1/\alpha} t^{-1/\alpha}C^k\sum_{j=1}^k\int_{0\leq s_1\leq \dots \leq s_{k-1}\leq t} \Big(\prod_{i\not=j, i\leq k}(s_i-s_{i-1})^{-1+\delta_\infty}\Big)(s_j-s_{j-1})\, ds_1, \dots ds_k
\\&\leq t^{-1/\alpha+\delta+(k-1)\delta_\infty}C^k k^{1/\alpha+1}{\Gamma(\delta_\infty)^{k-1}\Gamma(\delta)\over \Gamma((k-1)\delta_\infty+\delta)}. \ea
$$
Taking the sum in $k\geq 1$, we obtain the first bound in \eqref{full_Psi}.

We also have
$$
p_t^0(x,y)\leq Ct^{-1/\alpha}, \quad \sup_y\int_{\Re} p_t^0(x,y)\, dx\leq C.
$$
Repeating the calculation used in the proof of \eqref{full_Psi}, we get
$$
|r_t(x,y)|=\left(\int_0^{t/2}+\int_{t/2}^t\right) |(p_{t-s}^0\ast \Psi_s)(x,y)|\, ds\leq C 2^{1/\alpha}t^{-1/\alpha}t^{\delta_\infty} + C 2^{1+1/\alpha-\delta}t^{-1/\alpha+\delta}\leq  Ct^{-1/\alpha +\delta_\infty}.
$$
Combined with \eqref{R_tilde_bound}, this completes the proof.

\section{Proof of Theorem \ref{mainthm3}}\label{s7}

We further analyze the bound \eqref{Phi_bound_final} under the stronger assumption $\mathbf{H}^{\nu}$ (ii).   To simplify the notation, we assume $\gamma\leq \alpha$ and write $\gamma$ instead of $\gamma'$. This does not restrict generality because decreasing $\gamma$ in the assumption \eqref{reg_cond} leaves this assumption true. We have
$$\ba
Q_t(x,y)&=t^{\beta/\alpha}|\Phi^{\nu,large,+}_t(x,y)|
\\&\leq Ct^{\beta/\alpha}|\int_{|u|>t^{1/\alpha}}t^{-1/\alpha}G^{(\alpha)}\left({\kappa_t(y)-x-u\over t^{1/\alpha}}\right)\Big(|u|^{-\beta-1}1_{|u|\leq 1}+|u|^{-\gamma-1}1_{|u|>1}\Big)\, du
\\&=C\int_{|u|>t^{1/\alpha}}G_t^{(\alpha,\alpha, \alpha)}(x+u,\kappa_t(y))G_t^{(\alpha, \beta, \gamma)}(0,u)\, du
\\&\leq C\Big(G_t^{(\alpha,\alpha, \alpha)}\ast G_t^{(\alpha, \beta, \gamma)}\Big)(x,\kappa_t(y)),
\ea
$$
in the last inequality we have used that $G_t^{(\alpha,\alpha, \alpha)}(x,y)$ and  $G_t^{(\alpha,\beta, \gamma)}(x,y)$ actually depend on $|x-y|$, only.
Recall that $\beta<\alpha, \gamma\leq \alpha$. Then it is a direct calculation to check that
$$
G_t^{(\alpha,\alpha, \alpha)}(x,y)\leq G_t^{(\alpha,\beta, \gamma)}(x,y),\quad
G_{2t}^{(\alpha,\beta, \gamma)}(x,y)\leq CG_t^{(\alpha,\beta, \gamma)}(x,y).
$$
Then it follows from the sub-convolution property for $G_t^{(\alpha,\beta, \gamma)}(x,y)$ (see Appendix \ref{sA4}) that
$$
Q_t(x,y)\leq C G_t^{(\alpha, \beta, \gamma)}(x,\kappa_t(y)).
$$
That is, by \eqref{Phi_bound_final} we have
$$\ba
|\Phi_t(x,y)|&\leq Ct^{-1+\delta}G^{(\alpha, \alpha,\alpha)}_t(x,\kappa_t(y))
\\&+Ct^{-1+\delta_\zeta}G^{(\alpha, \alpha-\zeta,\alpha)}_t(x,\kappa_t(y))+Ct^{-1+\delta_\beta}G^{(\alpha, \beta,\gamma)}_t(x,\kappa_t(y)).
\ea
$$
Since
$$
t^{-1+\delta_\zeta}G^{(\alpha, \alpha-\zeta,\alpha)}_t(x,y) =\left\{                                 \begin{array}{ll}{t^{-1/\alpha-(\alpha-\zeta)/\alpha}},& |y-x|\leq t^{1/\alpha}\\
                                     |y-x|^{-(\alpha-\zeta)-1}, & t^{1/\alpha}<|y-x|\leq 1,\\
                                 |y-x|^{-\alpha-1}, & |y-x|>1,
                                    \end{array}
                                  \right.
$$
$$
t^{-1+\delta_\beta}G^{(\alpha, \beta, \gamma)}_t(x,y)= \left\{                                 \begin{array}{ll}{t^{-1/\alpha-\beta/\alpha}},& |y-x|\leq t^{1/\alpha}\\
                                    |y-x|^{-\beta-1}, & t^{1/\alpha}<|y-x|\leq 1,\\
                                  |y-x|^{-\gamma-1}, & |y-x|>1,
                                    \end{array}
   \right.
$$
 the sum of these kernels satisfies
$$\ba
t^{-1+\delta_\zeta}G^{(\alpha, \alpha-\zeta,\alpha)}_t(x,y)+t^{-1+\delta_\beta}G^{(\alpha, \beta, \gamma)}_t(x,y)&\leq 2\left\{                                 \begin{array}{ll}{t^{-1/\alpha-\beta'/\alpha}},& |y-x|\leq t^{1/\alpha}\\
                                    |y-x|^{-\beta'-1}, & t^{1/\alpha}<|y-x|\leq 1,\\
                                  |y-x|^{-\gamma-1}, & |y-x|>1,
\end{array}
   \right.
\\&
=2t^{-1+\delta'}G^{(\alpha, \beta', \gamma)}_t(x,y),
\ea
$$
where
$$
\beta'=\max(\beta, \alpha- \zeta), \quad \quad \delta'={\alpha-\beta'\over \alpha}>0
$$
(recall that we have assumed $\gamma\leq \alpha$). This finally gives the bound
\be\label{Phi_bound_sharp}
|\Phi_t(x,y)|\leq Ct^{-1+\delta}H^1_t(x,y)+Ct^{-1+\delta'} H^2_t(x,y)
\ee
with
\be\label{kernels}
H^1_t(x,y)=G^{(\alpha, \alpha,\alpha)}_t(x,\kappa_t(y))={1\over t^{1/\alpha}}G^{(\alpha)}\left(\kappa_t(y)-x\over t^{1/\alpha}\right), \quad H^2_t(x,y)=G^{(\alpha, \beta', \gamma)}_t(x,\kappa_t(y)).
\ee

Denote $\delta_1=\delta, \delta_2=\delta'$. For any $k>1$ we have
$$\ba
|\Phi_t^{\star k}(x,y)|\leq& C^k\sum_{i_1, \dots, i_k\in \{1,2\}}\int_{0<s_1<\dots<s_{k-1}<t}\left(\prod_{j=1}^{k}(s_j-s_{j-1})^{-1+\delta_{i_j}}\right)\times
\\&\hspace{4cm}\times\Big(H^{i_1}_{s_1}\ast\dots\ast H^{i_k}_{t-s_k}\Big)(x,y)\, ds_1, \dots, ds_{k-1}.
\ea
$$
The kernels $H^1, H^2$ satisfy $
H^1_t(x,y)\leq H^2_t(x,y)$ and have the sub-convolution property, see Proposition \ref{pAsub_conv_flow} below.
Then   for $t\in (0,T]$
$$\ba
\int_{0<s_1<\dots<s_{k-1}<t}&\left(\prod_{j=1}^{k}(s_j-s_{j-1})^{-1+\delta_{1}}\right)\Big(H^{1}_{s_1}\ast\dots\ast H^{1}_{t-s_k}\Big)(x,y)\, ds_1, \dots, ds_{k-1}
\\&\leq C^k H^1_t(x,y)\int_{0<s_1<\dots<s_{k-1}<t}\left(\prod_{j=1}^{k}(s_j-s_{j-1})^{-1+\delta_{1}}\right)\, ds_1, \dots, ds_{k-1}
\\&=t^{-1+k\delta_1}{C^k\Gamma(\delta_1)^k\over \Gamma(k\delta_1)} H^1_t(x,y)  \leq \wt C t^{-1+\delta_1}{C^k\Gamma(\delta_1)^k\over \Gamma(k\delta_1)} H^1_t(x,y),
\ea
$$
and   (recall that $\delta_1<\delta_2$)
$$\ba
\int_{0<s_1<\dots<s_{k-1}<t}&\left(\prod_{j=1}^{k}(s_j-s_{j-1})^{-1+\delta_{i_j}}\right)\Big(H^{i_1}_{s_1}\ast\dots\ast H^{i_k}_{t-s_k}\Big)(x,y)\, ds_1, \dots, ds_{k-1}
\\&\leq C^k H^2_t(x,y)\int_{0<s_1<\dots<s_{k-1}<t}\left(\prod_{j=1}^{k}(s_j-s_{j-1})^{-1+\delta_{i_j}}\right)\, ds_1, \dots, ds_{k-1}
\\&=t^{-1+\sum_j\delta_{i_j}}{C^k\prod_j\Gamma(\delta_{i_j})\over \Gamma(\sum_j\delta_{i_j})} H^2_t(x,y)  \leq \wt C t^{-1+\delta_2}{C^k\Gamma(\delta_2)^k\over \Gamma(k\delta_1)} H^2_t(x,y),
\ea
$$
if at least one of the indices $i_1, \dots, i_k$ equals 2. Thus
\be\label{Psi_bound_sharp}
|\Psi_t(x,y)|\leq \sum_{k\geq 1}|\Phi^{\star k}_t(x,y)|\leq Ct^{-1+\delta_1}H_{t}^1(x,y)+Ct^{-1+\delta_2}H_{t}^2(x,y).
\ee
Recall that $\delta_1=\delta, \delta_2=\delta'$ and $$
r_t(x,y)=(p\star \Psi)_t(x,y), \quad p_t^0(x,y)\leq C H_t^1(x,y).
$$
Then, using the sub-convolution properties of $H^1, H^2$ and the inequality $H^1\leq H^2$ in the same way we did before, we get
\be\label{r_bound_sharp}
|r_t(x,y)|\leq C\Delta(t)H^1_t(x,y)+Ct^{\delta'} H^2_t(x,y).
\ee
In the notation from the proof of Proposition \ref{pAsub_conv_flow}, we have
$$
H^1_t(x,y)=F_t^{(\alpha,\alpha,\alpha)}\left({x-\kappa_t(y)\over t^{1/\alpha}}\right), \quad H^2_t(x,y)=F_t^{(\alpha,\beta',\gamma)}\left({x-\kappa_t(y)\over t^{1/\alpha}}\right).
$$
Using \eqref{flowCorr} and \eqref{vagueF}, \eqref{F_mult}, we get
\be\label{r_bound_point}
|r_t(x,y)|\leq C\Delta(t)G^{(\alpha, \alpha, \alpha)}_t(\chi_t(x),y)+Ct^{\delta'}G^{(\alpha, \beta', \gamma')}_t(\chi_t(x),y).
\ee
Combined with \eqref{R_tilde_bound}, this completes the proof.

\appendix

\section{Appendix}\label{sA}
\subsection{Proof of \eqref{b-B}, \eqref{Lip_B}.}\label{sA1}

Denote
$$
N_\beta(\eps):=\left\{
                                                                     \begin{array}{ll}
                                                                               {1\over |1-\beta|}\eps^{1-\beta}, & \beta\in (0,2), \beta\not=1; \\
                                                                               1+\log \eps^{-1}, & \beta=1. \\
                                                                       \end{array}
                                                                           \right.
$$
The proof of the following statement is easy and omitted.

\begin{prop}\label{p2} Let $\upsilon(du)$ be a measure satisfying
$$
\upsilon(|u|>r) \leq C_\upsilon r^{-\beta}, \quad r\in (0, 1]
$$
for some $\beta\in (0,2)$. Then
\be\label{est_1_large}
\int_{\eps<|u|\leq 1} |u|\, \sigma(du) \leq C N_\beta(\eps), \quad \eps\leq 1
\ee
for $\beta\in [1,2)$, and
\be\label{est_1_small}
\int_{|u|\leq \eps} |u|\,\upsilon(du)\leq C N_\beta(\eps), \quad \eps\leq 1.
\ee
for $\beta\in (0,1)$.   In addition, for any $\beta\in (0,2)$
\be\label{est_2}
\int_{|u|\leq \eps} |u|^2\,\upsilon(du)\leq C \eps^{2-\beta}, \quad \eps\leq 1.
\ee
 The constants $C$ in  \eqref{est_1_large}, \eqref{est_1_small},   and \eqref{est_2}   depend on $\beta$ and $C_\upsilon$, only.
\end{prop}

\begin{prop}\label{p1} Let $f$ be such that for some $\sigma\in [0,1]$
\be\label{f_H_lin}
\|f\|_{H_{\sigma,loc}}:=\sup_{x\not=y, |x-y|\leq 1}{|f(x)-f(y)|\over |x-y|^\sigma}<\infty.
\ee
 Then for each $t\in (0,T]$
$$
F_t(x):={1\over 2\sqrt{\pi} t^{1/\alpha}}\int_\Re e^{-z^2t^{-2/\alpha}}f(x-z)d z
$$
satisfies
\be\label{A_f_bound}
\sup_x|f(x)-F_t(x)|\leq C_{\sigma, \alpha, T} t^{\sigma/\alpha}\|f\|_{H_{\sigma,loc}},
\ee
and  $F_t$ is Lipschitz continuous with $\mathrm{Lip}(f_t)\leq C_{\sigma, \alpha, T}t^{\sigma/\alpha-1/\alpha}\|f\|_{H_{\sigma,loc}}$.

\end{prop}
\begin{proof}   It follows from  \eqref{f_H_lin} that for $|x-y|\geq 1$
$$
|f(x)-f(y)|\leq 2|x-y|\|f\|_{H_{\sigma,loc}}.
$$
This inequality for \emph{large} $|x-y|$, combined with the inequality  \eqref{f_H_lin} for \emph{small} $|x-y|$ yields the following bound valid for all $x,y\in \Re$:
\be\label{f_H_lin1}
|f(x)-f(y)|\leq  2\Big(|x-y|^\sigma\vee|x-y|\Big)\|f\|_{H_{\sigma,loc}}\leq 2  \Big(|x-y|^\sigma+|x-y|\Big)\|f\|_{H_{\sigma,loc}}.
\ee
  Then
$$\ba
|F_t(x)-f(x)|&\leq  {1\over 2\sqrt{\pi} t^{1/\alpha}}\int_\Re e^{-z^2t^{-2/\alpha}}|f(x-z)-f(x)|d z
\\& \leq  {1\over 2\sqrt{\pi} t^{1/\alpha}}\|f\|_{H_{\sigma,loc}}\int_\Re e^{-z^2t^{-2/\alpha}}\Big(C|z|^\sigma+C|z|\Big)d z=\Big(C_1 t^{\sigma/\alpha}+C_2t^{1/\alpha}\Big)\|f\|_{H_{\sigma,loc}},
\ea
$$
which proves \eqref{A_f_bound}. Since
$$\ba
\prt_x F_t(x)={1\over 2\sqrt{\pi} t^{1/\alpha}}\int_\Re e^{-z^2t^{-2/\alpha}}f'(x-z)\,d z&=-{1\over \sqrt{\pi} t^{1/\alpha}}\int_\Re zt^{-2/\alpha}e^{-z^2t^{-2/\alpha}}f(x-z)d z
\\&={1\over \sqrt{\pi} t^{1/\alpha}}\int_\Re zt^{-2/\alpha}e^{-z^2t^{-2/\alpha}}(f(x)-f(x-z))\, d z,
\ea
$$
similar calculation gives
$$
\mathrm{Lip}(F_t)\leq\sup_x|\prt_tf(x)|\leq \Big(C_1 t^{\sigma/\alpha-1}+C_2t^{1/\alpha-1}\Big)\|f\|_{H_{\sigma,loc}}.
$$
\end{proof}
Now we are ready to prove
\eqref{b-B}, \eqref{Lip_B}. We decompose
\be\label{b_decomp}
b_t(x)=\wt b(x)+\widetilde m_t^{(\alpha)}(x)+\widetilde m_t^{\nu}(x)=:f^{1}(x)+f^{2,t}(x)+f^{3,t}(x),
\ee
where we denote
$$
\widetilde m_t^{(\alpha)}(x)=\left\{
                                               \begin{array}{ll}
                                                 -\int_{ t^{1/\alpha}<|u|\leq 1}u\mu^{(\alpha)}(x;du), & \alpha\in [1,2), \\
                                                 \int_{|u|\leq t^{1/\alpha}}u\mu^{(\alpha)}(x;du), & \alpha\in (0,1),
                                               \end{array}
                                             \right.
$$
$$ \widetilde m_t^{\nu}(x)=\left\{
                                               \begin{array}{ll}
                                                 -\int_{ t^{1/\alpha}<|u|\leq 1}u\nu(x;du), & \beta\in [1,2), \\
                                                 \int_{|u|\leq t^{1/\alpha}}u\nu(x;du), & \beta\in (0,1).
                                               \end{array}
                                             \right.
$$
By condition \eqref{b_H_lin}, we have $\|f^1\|_{H_\eta, loc}\leq C$. By condition $\mathbf{H}^{(\alpha)}$(i) and Proposition \ref{p2}, we have
$\|f^{2,t}\|_{H_\zeta, loc}\leq CN_\alpha(t^{1/\alpha})$. Finally, by \eqref{cond_ups_small}, \eqref{cond_ups_large}, and  Proposition \ref{p2}, we have $\|f^{3,t}\|_{H_0, loc}\leq CN_\beta(t^{1/\alpha})$. Applying Proposition \ref{p1}, we get \eqref{b-B}, \eqref{Lip_B}:
$$
|b_t(x)-B_t(x)|\leq C t^{\eta/\alpha}+Ct^{\zeta/\alpha}N_\alpha(t^{1/\alpha})+CN_\beta(t^{1/\alpha})\leq Ct^{\delta},
$$
$$
\mathrm{Lip}\, (B_t)\leq C t^{\eta/\alpha-1}+Ct^{\zeta/\alpha-1}N_\alpha(t^{1/\alpha})+Ct^{-1}N_\beta(t^{1/\alpha})\leq Ct^{-1+\delta}.
$$

Note that the above calculation also gives
\be\label{delta_b}
|b_t(x)- \wt b(x)| \leq CN_{\alpha}(t^{1/\alpha})\leq Ct^{-1+1/\alpha}\Big(1+1_{\alpha=1}\log_+ t^{-1}\Big).
\ee

\subsection{Auxiliary family $\chi_s^t(x)$ and properties of $\chi_t(x)$, $\kappa_t(y)$}

\begin{prop}\label{pCflow} For any $T>0$ there exists $C>1$ such that
\be\label{flowC}
C^{-1}|x|\leq |\chi_t(x)|\leq  C|x|, \quad 0\leq  t\leq T, \quad |x|\geq C.
\ee
\end{prop}
\begin{proof}   By $\mathbf{H}^{drift}$, the function $\wt b$ satisfies \eqref{f_H_lin} with $\sigma=\eta$. Therefore for this function \eqref{f_H_lin1} with $\sigma=\eta$ holds . Then by \eqref{b-B} and \eqref{delta_b} the coefficient $B_{t}$ in the ODE, which defines $\chi_t$, satisfies the following linear growth bound:
$$
|B_{t}(x)|\leq C_1+C_2t^{-1+1/\alpha}\Big(1+1_{\alpha=1}\log_+ t^{-1}\Big)+C_3|x|.
$$
This in a standard way provides
$$
e^{-C_3}|x|-C_4\leq |\chi_t(x)|\leq e^{C_3}|x|+C_4.
$$
\end{proof}

In order to relate the families $\chi_s(x)$, $\kappa_s(y)$, we introduce an auxiliary family $\chi_s^t(x)$, the solution to the Cauchy problem
\be\label{chi_tt}
{d\over ds}\chi_s^t(x)=B_{t-s}(\chi_s^t(x)), \quad s\in [0, t], \quad \chi_0^t(x)=x.
\ee

\begin{prop}\label{pAflow} For any $T>0$ there exists $C$ such that for any $s\leq t\leq T$
\be\label{flowA}
e^{-Ct^\delta}|\kappa_{t}(y)-x|\leq  |\kappa_{t-s}(y)-\chi_s^t(x)|\leq e^{Ct^\delta}|\kappa_{t}(y)-x|.
\ee
\end{prop}
\begin{proof} Denote $x_s=\chi^t_s(x), y_s=\kappa_{t-s}(y)$, then
$$
(x_s-y_s)'=(x_s-y_s)q_{t,s}, \quad q_{t,s}={B_{t-s}(x_s)-B_{t-s}(y_s)\over x_s-y_s}
$$
with the convention ${0\over 0}=1$. Then
$$
\chi^t_s(x)-\kappa_{t-s}(y)=x_s-y_s=(x_0-y_0)\exp\left(\int_0^s q_{t,r}\, dr\right)=(x-\kappa_t(y))\exp\left(\int_0^s q_{t,r}\, dr\right),
$$
which provides the required statement by \eqref{Lip_B} since $
|q_{t,r}|\leq \mathrm{Lip}\, (B_{t-r}).
$
\end{proof}

\begin{prop}\label{pBflow} For any $T>0$ and $\delta< \min(\delta_\eta,\delta_\zeta,\delta_\beta)$ there exist $C$ such that for
\be\label{chi-t-chi}
 \chi_s^t(x)=\chi_s(x)+t^{1/\alpha}\int_0^s \upsilon(\chi_r(x))\Big(W_\alpha(t; r)-W_\alpha(t;t-r)\Big)\, dr+Q_{s,t}(x), \quad s\leq t
\ee
with
\be\label{chi-t-chi-error}
|Q_{s,t}(x)|\leq Ct^{1/\alpha+\delta}, \quad s\leq t\leq T.
\ee
\end{prop}
\begin{proof}
Denote $x_s=\chi^t_s(x), \widetilde x_s=\chi_s(x)$, then
$$
(x_s-\wt x_s)'=(x_s-\wt x_s)\wt q_{t,s}+\wt Q_{t,s}, \quad \wt q_{t,s}={B_{t-s}(x_s)-B_{t-s}(\wt x_s)\over x_s-\wt x_s}, \quad \wt Q_{t,s}=B_{t-s}(\wt x_s)-B_{s}(\wt x_s),
$$
and thus
\be\label{1}
x_s-\wt x_s=\int_0^s \wt Q_{t,r} \exp\left(\int_r^s \wt q_{t,w}\, dw\right)\, dr.
\ee
By  \eqref{Lip_B}
\be\label{2}
\left|\exp\left(\int_r^s \wt q_{t,w}\, dw\right)-1\right|\leq Ct^\delta.
\ee
On the other hand, by \eqref{b-B}
$$
|\wt Q_{t,r}- (b_{t-r}(\wt x_r)-b_{r}(\wt x_r))|\leq C\Big((t-r)^{-1+1/\alpha+\delta}+r^{-1+1/\alpha+\delta}\Big),
$$
 and by \eqref{m_tb_t}
$$
b_{t-r}(x)-b_{r}(x)=m_r^\mu(x)-m^\mu_{t-r}(x)=
\Big( m_{r}^{(\alpha)}(x)-m_{t-r}^{(\alpha)}(x)\Big)+\Big(m_{r}^{\nu}(x)-m_{t-r}^{\nu}(x)\Big),
$$
where we denote
$$
m_r^{(\alpha)}(x)=\int_{r^{1/\alpha<|u|\leq 1}}u\mu^{(\alpha)}(x,du),\quad  m_r^{\nu}(x)=\int_{r^{1/\alpha<|u|\leq 1}}u\nu(x,du).
$$
Assume for a while that $r\leq t-r$. By Proposition \ref{p2},
$$
|m_{r}^{\nu}(x)-m_{t-r}^{\nu}(x)|\leq \int_{r^{1/\alpha}<|u|\leq (t-r)^{1/\alpha}}|u| \,|\nu|(x,du)\leq \left\{
                                                                                                          \begin{array}{ll}
                                                                                                            CN_\beta((t-r)^{1/\alpha}), & \beta\in (0,1); \\
                                                                                                           CN_\beta(r^{1/\alpha}), &  \beta\in [1, 2).
                                                                                                          \end{array}
                                                                                                        \right.
$$
  Similarly, for $t-r\leq r\leq t$
$$
|m_{r}^{\nu}(x)-m_{t-r}^{\nu}(x)|\leq \int_{(t-r)^{1/\alpha}<|u|\leq r^{1/\alpha} }|u| \,|\nu|(x,du)\leq \left\{
                                                                                                          \begin{array}{ll}
                                                                                                            CN_\beta(r^{1/\alpha}), & \beta\in (0,1); \\
                                                                                                           CN_\beta((t-r)^{1/\alpha}), &  \beta\in [1, 2).
                                                                                                          \end{array}
                                                                                                        \right.
$$
That is, in any case we have
$$
|m_{r}^{\nu}(x)-m_{t-r}^{\nu}(x)|\leq CN_\beta(r^{1/\alpha})+CN_\beta((t-r)^{1/\alpha})\leq C \Big((t-r)^{-1+1/\alpha+\delta}+r^{-1+1/\alpha+\delta}\Big).
$$
On the other hand, for $r\leq t$ we have
$$\ba
m_r^{(\alpha)}(x)&=m_t^{(\alpha)}(x)+\int_{r^{1/\alpha}<|u|\leq t^{1/\alpha}}u\mu^{(\alpha)}(x,du)
\\&=
m_t^{(\alpha)}(x)+\upsilon(x)\int_{r^{1/\alpha}}^{t^{1/\alpha}}{dw\over w^\alpha}=m_t^{(\alpha)}(x)+\upsilon(x)t^{1/\alpha}W(t;r).
\ea
$$
Summarizing these calculations we get
$$
\left|\wt Q_{t,r}- \upsilon(\wt x_r)t^{1/\alpha}\Big(W_\alpha(t;r)-W_\alpha(t;t-r)\Big)\right|\leq C\Big((t-r)^{-1+1/\alpha+\delta}+r^{-1+1/\alpha+\delta}\Big).
$$
This bound combined with  \eqref{1} and \eqref{2}, provides  \eqref{chi-t-chi} and \eqref{chi-t-chi-error}.
\end{proof}

Recall that $\upsilon(\cdot)$ is bounded and $W_\alpha(t;\cdot)$ is a probability density. That is, directly from \eqref{chi-t-chi}, \eqref{chi-t-chi-error} we get the bound
\be\label{chi-t-chi-error3}
|\chi_s^t(x)-\chi_s(x)|\leq Ct^{1/\alpha}, \quad s\leq t\leq T.
\ee
Combined with Proposition \ref{pAflow}, this gives the following.

\begin{cor}\label{corA} For each $T>0$, there exists $C$ such that for any $s\leq t\leq T$
\be\label{flowCorr}
\ba
 e^{-Ct^\delta}|\kappa_{t}(y)-x|-Ct^{1/\alpha}\leq  |\kappa_{t-s}(y)-\chi_s(x)|
\leq e^{Ct^\delta}|\kappa_{t}(y)-x|+Ct^{1/\alpha}, \quad s\leq t\leq T.
\ea
 \ee
\end{cor}

\subsection{Stable densities}\label{sA3}

The kernel $G^{(\alpha)}(x)$ (see the definition before Corollary \ref{cor_p}) possess the following   properties which can be  verified straightforwardly:
\be\label{G_comp}
G^{(\alpha)}(x)\leq G^{(\beta)}(x),\quad 0<\beta<\alpha;
\ee
\be\label{G_pol}
(1+|x|)^\beta G^{(\alpha)}(x)\leq C G^{(\alpha-\beta)}(x),\quad 0<\beta<\alpha;
\ee
\be\label{vague}
\sup_{|v|\leq 1}G^{(\alpha)}(x+v)\leq CG^{(\alpha)}(x),
\ee
and for any $c>0$ there exists $C$ such that
\be\label{G_mult}
G^{(\alpha)}(cx)\leq CG^{(\alpha)}(x).
\ee
We also have
\be\label{sub_conv_simple}
\Big(G^{(\alpha)}\ast G^{(\alpha)}\Big)(x)\leq CG^{(\alpha)}(x).
\ee

The following two propositions collect the properties of the $\alpha$-stable densities $g^{(\lambda,\rho,\upsilon)}(x)$, see \eqref{stable_den} for the definition.

\begin{prop}\label{pA4} The density $g^{(\lambda,\rho,\upsilon)}(x)$ is well defined and belongs to the class $C^1$ w.r.t. $(\lambda, \rho),$ and the class $C^2_\infty$ w.r.t. $x$. The following bounds hold true for each $\alpha\in (0,2), 0<\lambda_{\min}\leq \lambda_{\max}, R>0$  uniformly in $\lambda\in [\lambda_{\min}, \lambda_{\max}]$, $\rho\in [-1,1]$, $|\upsilon|\leq R$, $x\in \Re$:
\be\label{g_bound0}
{g}^{(\lambda,\rho,\upsilon)}(x)\leq C G^{(\alpha)}(x),
\ee
\be\label{g_bound_lamb}
|\prt_\lambda{g}^{(\lambda,\rho,\upsilon)}(x)|\leq C G^{(\alpha)}(x),
\ee
\be\label{g_bound_rho}
|\prt_\rho{g}^{(\lambda,\rho,\upsilon)}(x)|\leq C G^{(\alpha)}(x),
\ee
\end{prop}
\begin{proof}
 We use the standard trick of a decomposition of an infinitely divisible law into a convolution of `small' and `large' jump parts.  Namely, we put
$$
\Psi^{(\lambda,\rho,\upsilon)}_{\alpha,small}(\xi)=i\xi\upsilon+ \int_{|u|\leq 1}\Big(e^{iu\xi}-1-iu\xi\Big)\mu^{(\alpha,\lambda,\rho)}(du),\quad \Psi^{(\lambda,\rho)}_{\alpha,large}(\xi)= \int_{|u|> 1}\Big(e^{iu\xi}-1\Big)\mu^{(\alpha,\lambda,\rho)}(du),
$$
and observe that $\exp \Psi^{(\lambda,\rho)}_{\alpha,large}(\xi)$ is the Fourier transform of a compound Poisson process with the intensity of the `Poisson clock' equal
$$
\int_{|u|>1}\mu^{(\alpha,\lambda,\rho)}(du)={2\lambda\over \alpha},
$$
and with the law of a single jump having the density
$$
m^{(\alpha,\rho)}(u)={\alpha(1+\rho\, \sgn\, u)\over 2|u|^{\alpha+1}}1_{|u|>1}.
$$
On the other hand,
$$
\Psi^{(\lambda,\rho,\upsilon)}_{\alpha,small}(\xi)= i\xi\upsilon-\lambda \int_{|u|\leq 1}{1-\cos (u\xi)\over |u|^{\alpha+1}}\,du-i\lambda\rho \int_{|u|\leq 1}{u\xi\sin(u\xi)\over |u|^{\alpha+1}}\,du,
$$
and in particular
\be\label{Cau_1}
\mathrm{Re}\, \Psi^{(\lambda,\rho,\upsilon)}_{\alpha,small}(\xi)=-\lambda \int_{|u|\leq 1}{1-\cos (u\xi)\over |u|^{\alpha+1}}\,du\leq -c_1|\xi|^\alpha+c_2
\ee
with some positive $c_1, c_2$. This yields that the inverse Fourier transform $$
f^{(\lambda,\rho,\upsilon)}(x)={1\over 2\pi}\int_{\Re}e^{-ix\xi+\Psi^{(\lambda,\rho,\upsilon)}_{\alpha,small}(\xi)}\, d\xi
$$
is well defined. Then the  density $g^{(\lambda,\rho,\upsilon)}(x)$ is also well defined and possesses the representation
\be\label{g_rep}
g^{(\lambda,\rho,\upsilon)}(x)=e^{-2\lambda/\alpha}f^{(\lambda,\rho,\upsilon)}(x)+e^{-2\lambda/\alpha}\sum_{k=1}^\infty  {1\over k!}\Big(f^{(\lambda,\rho,\upsilon)}\ast [\widetilde m^{(\lambda,\rho,\upsilon)}]^{\ast k}\Big)(x)
\ee
with
$$
\widetilde m^{(\alpha,\lambda,\rho)}(\xi)=(2\lambda/\alpha)m^{(\alpha,\rho)}(\xi).
$$
We claim that, uniformly in $\lambda\in [\lambda_{\min}, \lambda_{\max}]$, $\rho\in [-1,1]$, $|\upsilon|\leq R$, $x\in \Re$,
\be\label{f_bounds}
|{f}^{(\lambda,\rho,\upsilon)}(x)|+|\prt_\lambda{f}^{(\lambda,\rho,\upsilon)}(x)|+|\prt_\rho{f}^{(\lambda,\rho,\upsilon)}(x)|
+|\prt_x{f}^{(\lambda,\rho,\upsilon)}(x)|+|\prt_{xx}^2{f}^{(\lambda,\rho,\upsilon)}(x)|\leq C e^{-|x|}.
\ee
The argument here is quite standard (e.g. \cite{KK13}), but for the sake of completeness we outline  the proof. The function $\Psi^{(\lambda,\rho,\upsilon)}_{\alpha, small}$ is defined as an integral over the bounded interval $[-1,1]$, and thus has an analytic extension to $\mathbb{C}$:
$$\ba
\Psi^{(\lambda,\rho,\upsilon)}_{small}(\xi+i\phi)&=i\xi\upsilon-\phi\upsilon+\int_{|u|\leq 1}(e^{-u\phi+iu\xi}-1-iu\xi+u\phi)\,\mu^{(\alpha,\lambda,\rho)}(du)
\\&=\int_{|u|\leq 1}e^{-u\phi}(e^{iu\xi}-1-iu\xi)\,\mu^{(\alpha,\lambda,\rho)}(du)+\int_{|u|\leq 1}(e^{-u\phi}-1+ u\phi)\,\mu^{(\alpha,\lambda,\rho)}(du)
\\&\hspace{2cm}-\phi\upsilon+i\xi\upsilon+i\int_{|u|\leq 1}(e^{-u\phi}-1)u\xi\,\mu^{(\alpha,\lambda,\rho)}(du).
\ea
$$
Then for $\phi\in [-1,1]$ we have
\be\label{Cau_2}\ba
\mathrm{Re}\,\Psi^{(\lambda,\rho,\upsilon)}_{\alpha,small}(\xi+i\phi)&=-\int_{|u|\leq 1}e^{-u\phi}(1-\cos(u\xi))\,\mu^{(\alpha,\lambda,\rho)}(du)
\\&+\int_{|u|\leq 1}(e^{-u\phi}-1+ u\phi)\,\mu^{(\alpha,\lambda,\rho)}(du)-\phi\upsilon
\\&\leq -e^{-1}\lambda \int_{|u|\leq 1}{1-\cos (u\xi)\over |u|^{\alpha+1}}\,du
\leq -e^{-1}(c_1|\xi|^\alpha+c_2), \quad \xi\in \Re,
\ea
\ee
see \ref{Cau_1}. This makes it possible to change the integration contour in the  inverse Fourier transform formula from $\Re=\Re+i0$ to $\Re+i\phi$, which gives
$$
f^{(\lambda,\rho,\upsilon)}(x)={1\over 2\pi}\int_{\Re}e^{-ix\xi+x\phi+\Psi^{(\lambda,\rho,\upsilon)}_{\alpha, small}(\xi+i\phi)}\, d\xi.
$$
Take $\phi=\phi_x=-\mathrm{sgn}\, x$, then
$$
f^{(\lambda,\rho,\upsilon)}(x)={e^{-|x|}\over 2\pi}\int_{\Re}e^{-ix\xi+\Psi^{(\lambda,\rho,\upsilon)}_{\alpha,small}(\xi+i\phi_x)}\, d\xi.
$$
This representation and \eqref{Cau_2} give
$$\ba
{f}^{(\lambda,\rho,\upsilon)}(x)&+|\prt_\lambda{f}^{(\lambda,\rho,\upsilon)}(x)|+|\prt_\rho{f}^{(\lambda,\rho,\upsilon)}(x)|+|\prt_x{f}^{(\lambda,\rho,\upsilon)}(x)|+|\prt_{xx}^2{f}^{(\lambda,\rho,\upsilon)}(x)|
\\&
\leq Ce^{-|x|}\int_{\Re}e^{\mathrm{Re}\,\Psi^{(\lambda,\rho,\upsilon)}_{\alpha,small}(\xi+i\phi_x)}\Big(1+|\prt_\lambda \Psi^{(\lambda,\rho,\upsilon)}_{\alpha,small}(\xi+i\phi_x)|+|\prt_\rho \Psi^{(\lambda,\rho,\upsilon)}_{\alpha,small}(\xi+i\phi_x)|+|\xi|+\xi^2\Big)\, d\xi.
\ea
$$
It is easy to check that
$$
|\prt_\lambda \Psi^{(\lambda,\rho,\upsilon)}_{\alpha,small}(\xi+i\phi_x)|+|\prt_\rho \Psi^{(\lambda,\rho,\upsilon)}_{\alpha,small}(\xi+i\phi_x)|\leq C(1+\xi^2),
$$
hence  \eqref{f_bounds} follows by \eqref{Cau_2}.

Next, we give explicitly the function $\widetilde m^{(\alpha,\lambda,\rho)}(u)$ and its derivatives:
$$
\widetilde m^{(\alpha,\lambda,\rho)}(u)=\lambda{1+\rho\, \sgn\, u\over |u|^{\alpha+1}}1_{|u|>1},
$$
$$
\prt_\lambda \widetilde m^{(\alpha,\lambda,\rho)}(u)={1+\rho\, \sgn\, u\over |u|^{\alpha+1}}1_{|u|>1},\quad \prt_\rho\widetilde m^{(\alpha,\lambda,\rho)}(u)={\lambda\, \sgn\, u\over |u|^{\alpha+1}}1_{|u|>1},
$$
and observe that the absolute values of these functions are dominated by $CG^{(\alpha)}(u)$.

Now we can finalize the proof.  It follows from \eqref{f_bounds} that
$$
{f}^{(\lambda,\rho,\upsilon)}(x)\leq C G^{(\alpha)}(x),
$$
then taking $C$ large enough we obtain inductively
$$
\Big(f^{(\lambda,\rho,\upsilon)}\ast [\widetilde m^{(\alpha,\lambda,\rho)}]^{\ast k}\Big)(x)\leq C^{2k+1}G^{(\alpha)}(x),
$$
and applying \eqref{g_rep} we complete the proof of \eqref{g_bound0}. The proofs of \eqref{g_bound_lamb}, \eqref{g_bound_rho} are essentially the same. The minor difference is that respective derivatives of $
f^{(\lambda,\rho,\upsilon)}\ast [\widetilde m^{(\alpha,\lambda,\rho)}]^{\ast k}$
now actually contain $(k+1)$ summands, each of them being a $(k+1)$-fold convolution where each term is dominated by $CG^{(\alpha)}(x)$; however the extra  multiplier $(k+1)$ is not essential thanks to the term $1/k!$ in \eqref{g_rep}.
\end{proof}

\begin{prop}\label{pA5} The density $g^{(\lambda,\rho,\upsilon)}(x)$ belongs to the class $C^2_\infty$ w.r.t. $x$. The following bounds hold true for each $\alpha\in (0,2), 0<\lambda_{\min}\leq \lambda_{\max}, R>0$  uniformly in $\lambda\in [\lambda_{\min}, \lambda_{\max}]$, $\rho\in [-1,1]$, $|\upsilon|\leq R$, $x\in \Re$:
\be\label{g_bound_x}
|\prt_x{g}^{(\lambda,\rho,\upsilon)}(x)|\leq C G^{(\alpha+1)}(x),
\ee
\be\label{g_bound_xx}
|\prt^2_{xx}{g}^{(\lambda,\rho,\upsilon)}(x)|\leq C G^{(\alpha+2)}(x).
\ee
\be\label{g_bound_Lsym}
|L_x^{(\alpha),sym}{g}^{(\lambda,\rho,\upsilon)}(x)|\leq C G^{(\alpha)}(x),
\ee
\be\label{g_bound_Lasym}
|L_x^{(\alpha),asym}{g}^{(\lambda,\rho,\upsilon)}(x)|\leq C G^{(\alpha)}(x),
\ee
see \eqref{Lsym}, \eqref{Lasym} for the definition of $L^{(\alpha),sym}, L_x^{(\alpha),sym}$.
\end{prop}
\begin{proof} The proofs of \eqref{g_bound_Lsym}, \eqref{g_bound_Lasym} are completely analogous to the previous proof. Namely,  using \eqref{f_bounds} it is easy to verify that
$$
|L_x^{(\alpha),sym}{f}^{(\lambda,\rho,\upsilon)}(x)|+|L_x^{(\alpha),asym}{f}^{(\lambda,\rho,\upsilon)}(x)|\leq C G^{(\alpha)}(x),
$$
Then the required bounds follow by \eqref{g_rep} and \eqref{sub_conv_simple}. The new difficulty in \eqref{g_bound_x}, \eqref{g_bound_xx} is that the kernels in the right hand sides has the higher order of decay in $x$, and thus cannot be derived simply by \eqref{sub_conv_simple}.  We will prove the first of these inequalities only: the second one is quite analogous, though the calculation is more cumbersome. Like we did that in the previous proof, we use representation \eqref{g_rep} and analyze the derivatives of the terms in the right hand side sum. Note that by \eqref{Cau_1} the derivative $\prt_x f^{(\lambda,\rho,\upsilon)}(x)$ is well defined and
$$
\prt_x f^{(\lambda,\rho,\upsilon)}(x)={1\over 2\pi}\int_{\Re}(-i\xi)e^{-ix\xi+\Psi^{(\lambda,\rho,\upsilon)}_{\alpha, small}(\xi)}\, d\xi.
$$
Similarly to the previous proof, we deduce that
$$
|\prt_x f^{(\lambda,\rho,\upsilon)}(x)|\leq C e^{-|x|}\leq CG^{(\alpha+1)}(x).
$$
On the other hand, for $u\not=\pm1$ there exists
$$
\prt_u \widetilde m^{(\alpha,\lambda,\rho)}(u)=-(\alpha+1)\lambda{\sgn\, u+\rho \over |u|^{\alpha+2}}1_{|u|>1},
$$
and the absolute value of the latter function is dominated by $CG^{(\alpha+1)}(u)$.

Let us prove the following: there exists $C$ such that for any $f\in C^1$ with $|f(x)|\leq C_f G^{(\alpha)}(x)$, $|\prt_x f(x)|\leq C_f G^{(\alpha+1)}(x)$ the following inequalities hold:
\be\label{iterate}
|(f\ast \widetilde m^{(\alpha,\lambda,\rho)})(x)|\leq C C_f G^{(\alpha)}(x),\quad |\prt_x(f\ast \widetilde m^{(\alpha,\lambda,\rho)})(x)|\leq C C_f G^{(\alpha+1)}(x).
\ee
This will easily  yield
$$
|\prt_x\Big(f^{(\lambda,\rho,\upsilon)}\ast [\widetilde m^{(\alpha,\lambda,\rho)}]^{\ast k}\Big)(x)|\leq C^{k+1} G^{(\alpha+1)}(x)
$$
and complete the proof.

The first inequality in \eqref{iterate} follows just from  \eqref{sub_conv_simple}. To prove the second inequality, we first note that \eqref{sub_conv_simple} also yields that, for some  $C$,
$$
|\prt_x(f\ast \widetilde m^{(\alpha,\lambda,\rho)})(x)|=|(f'\ast \widetilde m^{(\alpha,\lambda,\rho)})(x)|
$$
is dominated by $C C_f G^{(\alpha)}(x)$. Hence it is sufficient to consider the case $|x|>2$, only. Let $x>2$, the case $x<-2$ is quite analogous. We have
$$\ba
\prt_x(f\ast \widetilde m^{(\alpha,\lambda,\rho)})(x)&=\int_{-\infty}^{x/2}f'(x-u)\widetilde m^{(\alpha,\lambda,\rho)}(u)\, du+\int_{x/2}^{\infty}f'(x-u)\widetilde m^{(\alpha,\lambda,\rho)}(u)\, du
\\&=\int_{-\infty}^{x/2}f'(x-u)\widetilde m^{(\alpha,\lambda,\rho)}(u)\, du+\int_{x/2}^{\infty}f(x-u)\Big(\widetilde m^{(\alpha,\lambda,\rho)}\Big)'(u)\, du
\\&-f(x/2)\widetilde m^{(\alpha,\lambda,\rho)}(x/2);
\ea
$$
note that $\widetilde m^{(\alpha,\lambda,\rho)}$ is smooth on $[x/2, \infty)\subset (1,\infty)$, hence we can apply the integration by parts formula here. The term $f'(x-u)$ in the first integral and the term $\Big(\widetilde m^{(\alpha,\lambda,\rho)}\Big)'(u)$ in the second integral are dominated by $C  G^{(\alpha+1)}(x/2)$. Then both these integrals are dominated by  $C  C_f G^{(\alpha+1)}(x/2)$, and we complete the proof of \eqref{iterate} using \eqref{G_mult}.
\end{proof}

\subsection{Properties of  $p_t^0(x,y)$: proofs of \eqref{Q1}, \eqref{arrow_id}, and \eqref{R_tilde_bound}.}\label{sA41}

\begin{proof}[Proof of \eqref{Q1}]
We decompose  $$
\ba
\Psi_\alpha(t,z;\xi)=\int_0^t \int_{\Re}\Big(e^{iu\xi}-1-iu\xi 1_{|u|\leq t^{1/\alpha}}\Big)&\mu^{(\alpha)}(\kappa_s(z);du)\, ds
\\&+
i\xi\int_0^t \int_{s^{1/\alpha}<|u|\leq t^{1/\alpha}}u\mu^{(\alpha)}(\kappa_s(z);du)ds.
\ea
$$
 Since the density of  $\mu^{(\alpha)}(\kappa_s(z);du)$ is a homogeneous function of $u$ of the order $-(\alpha+1)$, changing the variables  $v=ut^{-1/\alpha}$ we get
$$
\ba
\int_0^t& \int_{\Re}\Big(e^{iu\xi}-1-iu\xi 1_{|u|\leq t^{1/\alpha}}\Big)\mu^{(\alpha)}(\kappa_s(z);du)\, ds
\\&={1\over t}\int_0^t \int_{\Re}\Big(e^{ivt^{1/\alpha}\xi}-1-ivt^{1/\alpha}\xi 1_{|v|\leq 1}\Big)\mu^{(\alpha)}(\kappa_s(z);dv)\, ds
=\Psi^{(\wt\lambda_t(z),\wt \rho_t(z), 0)}_\alpha(t^{1/\alpha}\xi),
\ea
$$
see \eqref{stable_psi} for the definition of $\Psi^{\lambda, \rho, \upsilon}_\alpha$.
On the other hand, we have straightforwardly
$$
\int_0^t \int_{s^{1/\alpha}\leq|u|\leq t^{1/\alpha}}u\mu^{(\alpha)}(\kappa_s(z);du)ds=\int_0^t \upsilon(\kappa_s(z))\int_{s^{1/\alpha}}^{ t^{1/\alpha}}r\,{dr\over r^{\alpha+1}}ds=t^{1/\alpha}\wt \upsilon_t(z),
$$
see Section \ref{s41} for the definition of $\wt \upsilon_t$.
Thus
$$
\Psi_\alpha(t,z;\xi)=i\xi t^{1/\alpha}\wt \upsilon_t(z)+\Psi^{(\wt\lambda_t(z),\wt \rho_t(z), 0}_\alpha(t^{1/\alpha}\xi)=\Psi^{(\wt\lambda_t(z),\wt \rho_t(z), \wt \upsilon_t(z))}_\alpha(t^{1/\alpha}\xi).
$$
Therefore
$$\ba
h^{t, z}(w)&={1\over 2\pi}\int_\Re \exp\left[-iw\xi+\Psi_\alpha(t,z;\xi)\right]\, d\xi
\\&={1\over 2\pi}\int_\Re \exp\left[-iw\xi+\Psi^{(\wt\lambda_t(z),\wt \rho_t(z), \wt \upsilon_t(z))}_\alpha(t^{1/\alpha}\xi)\right]\, d\xi
=t^{-1/\alpha} g^{(\wt\lambda_t(z),\wt \rho_t(z), \wt \upsilon_t(z))}\left(w\over t^{1/\alpha}\right),
\ea
$$
which yields \eqref{p^01}.
\end{proof}

\begin{proof}[Proof of \eqref{arrow_id}.] Denote
$$
\psi_\alpha(t,z;\xi)= \prt_t\Psi_\alpha(t,z;\xi)= \int_{\Re}\Big(e^{iu\xi}-1-iu\xi 1_{|u|\leq t^{1/\alpha}}\Big)\mu^{(\alpha)}(\kappa_t(z);du).
$$
It is easy to show that $|\psi_\alpha(t,z;\xi)|\leq C(1+\xi^2)$. On the other hand, similarly to \eqref{Cau_1}, we have that for any $0<\tau<T$ there exist constants $c_1, c_2$ such that
$$
\mathrm{Re}\,\Psi_\alpha(t,z;\xi)\leq -c_1|\xi|^\alpha+c_2, \quad t\in [\tau, T].
$$
Then the dominated convergence gives
$$
\prt_t h^{t, z}(w-x)=\lim_{R\to \infty}{1\over 2\pi}\int_{-R}^R \psi_\alpha(t,z;\xi)e^{-iw\xi+ix\xi+\Psi_\alpha(t,z;\xi)}\, d\xi={1\over 2\pi}\int_\Re \psi_\alpha(t,z;\xi)e^{-iw\xi+ix\xi+\Psi_\alpha(t,z;\xi)}\, d\xi.
$$
for $t\in [\tau, T]$. Repeating the same argument, we get
$$\ba
\Lba^{(\alpha),z,t}_xh^{t, z}(w-x)&={1\over 2\pi}\int_\Re e^{-iw\xi+\Psi_\alpha(t,z;\xi)} \Big(\Lba^{(\alpha),z,t}_xe^{ix\xi}\Big)\, d\xi
={1\over 2\pi}\int_\Re e^{-iw\xi+\Psi_\alpha(t,z;\xi)} \psi_\alpha(t,z;\xi) e^{ix\xi}\, d\xi,
\ea
$$
$t\in [\tau, T]$, which proves \eqref{arrow_id} for these values of $t$. Since $0<\tau<T$ are arbitrary, this completes the proof.
\end{proof}

\begin{proof}[Proof of \eqref{R_tilde_bound}] Denote
$$
\overleftarrow\upsilon_t(x)=\int_0^t\upsilon(\chi_s(x))W_\alpha(t;t-s)\, ds.
$$
We have
$$\ba
p_t^0(x,y)-  p_t^{\mathrm{main}}(x,y)  &=\left({1\over t^{1/\alpha}} g^{(\wt \lambda_t(y), \wt \rho_t(y), \wt \upsilon_t(y))}\left({\kappa_t(y)-x\over t^{1/\alpha}}\right)-{1\over t^{1/\alpha}} g^{(\lambda_t(x), \rho_t(x), \overleftarrow\upsilon_t(x))}\left({\kappa_t(y)-x\over t^{1/\alpha}}\right)\right)
\\&+
\left({1\over t^{1/\alpha}} g^{(\lambda_t(x), \rho_t(x), \overleftarrow\upsilon_t(x))}\left({\kappa_t(y)-x\over t^{1/\alpha}}\right)-{1\over t^{1/\alpha}}g^{(\lambda_t(x), \rho_t(x), \upsilon_t(x))}\left({y-\chi_t(x)\over t^{1/\alpha}}\right)\right)
\\&=:R^1_t(x,y)+R_t^2(x,y).
\ea
$$
Note that
$$
{g}^{(\lambda,\rho,\upsilon)}(w)={g}^{(\lambda,\rho,0)}(w-\upsilon).
$$
Then by \eqref{g_bound_lamb}, \eqref{g_bound_rho}, and \eqref{g_bound_x} for any $0<\lambda_{\min}\leq \lambda_{\max}$, $R>0$ there exists $C$ such that for any $\lambda_{1}, \lambda_2\in [\lambda_{\min},  \lambda_{\max}],$ $\rho_1,\rho_2\in [-1, 1]$, $\upsilon_1,\upsilon_2\in [-R, R],$ and $x\in \Re$
\be\label{hhh}
|{g}^{(\lambda_1,\rho_1, \upsilon_1)}(x)-{g}^{(\lambda_2,\rho_2,\upsilon_2)}(x)|\leq C\Big(|\lambda_1-\lambda_2|+|\rho_1-\rho_2|+|\upsilon_1-\upsilon_2|\Big)G^{(\alpha)}(x).
\ee
We have
$$
|\wt \lambda_t(y)-\lambda_t(x)|=\left|{1\over t}\int_0^t\lambda(\kappa_{\tau}(y))\, d\tau-{1\over t}\int_0^t\lambda(\chi_s(x))\, ds\right|
=\left|{1\over t}\int_0^t\Big(\lambda(\kappa_{t-s}(y))-\lambda(\chi_s(x))\Big)\, ds\right|,
$$
in the last identity we changed the variable $\tau=t-s$. By \eqref{flowCorr},
$$
|\kappa_{t-s}(y)-\chi_s(x)|\leq C|y-\chi_t(x)|+Ct^{1/\alpha}, \quad s\in [0,t].
$$
Since function  $\lambda(\cdot)$ is $\zeta$-H\"older continuous and bounded, this gives
$$
|\wt \lambda_t(y)-\lambda_t(x)|\leq  C\Big(t^{\delta_\zeta}+|y-\chi_t(x)|^\zeta\wedge 1\Big).
$$
Similarly, $$
|\wt \rho_t(y)-\rho_t(x)|\leq  C\Big(t^{\delta_\zeta}+|y-\chi_t(x)|^\zeta\wedge 1\Big).
$$
Finally,
$$\ba
|\wt \upsilon_t(y)-\overleftarrow\upsilon_t(x)|&=\left|\int_0^t\upsilon(\kappa_\tau(y)) W_\alpha(t;\tau)\, d\tau-\int_0^t\upsilon(\chi_s(x))W_\alpha(t;t-s)\, ds\right|
\\&\leq \int_0^t|\upsilon(\kappa_{t-s}(y))- \upsilon(\chi_s(x))|W_\alpha(t;t-s)\, ds
\leq  C\Big(t^{\delta_\zeta}+|y-\chi_t(x)|^\zeta\wedge 1\Big).
\ea
$$
Thus by \eqref{hhh}
$$
| R^1_t(x,y)|\leq  C\Big(t^{\delta_\zeta}+|y-\chi_t(x)|^\zeta\wedge 1\Big)\left({1\over t^{1/\alpha}}G^{(\alpha)}\left({\kappa_t(y)-x\over t^{1/\alpha}}\right)\right).
$$
By  \eqref{flowCorr} and \eqref{vague}, \eqref{G_mult},
$$
{1\over t^{1/\alpha}}G^{(\alpha)}\left({\kappa_t(y)-x\over t^{1/\alpha}}\right)\leq C{1\over t^{1/\alpha}}G^{(\alpha)}\left({y-\chi_t(x)\over t^{1/\alpha}}\right)=G_t^{(\alpha, \alpha, \alpha)}(\chi_t(x), y).
$$
This gives finally
$$
|R^1_t(x,y)|\leq C\Big(t^{\delta_\zeta}+|y-\chi_t(x)|^\zeta\wedge 1\Big)G_t^{(\alpha, \alpha, \alpha)}(\chi_t(x), y)\leq Ct^{\delta_\zeta}G_t^{(\alpha,\alpha-\zeta,\alpha)}(\chi_t(x), y).
$$
Next, we decompose
$$\ba
R_t^2(x,y)&=\left({1\over t^{1/\alpha}} g^{(\lambda_t(x), \rho_t(x), \overleftarrow\upsilon_t(x))}\left({\kappa_t(y)-x\over t^{1/\alpha}}\right)-{1\over t^{1/\alpha}} g^{(\lambda_t(x), \rho_t(x), \overleftarrow\upsilon_t(x))}\left({y-\chi_t^t(x)\over t^{1/\alpha}}\right)\right)
\\&+\left({1\over t^{1/\alpha}} g^{(\lambda_t(x), \rho_t(x), \overleftarrow\upsilon_t(x))}\left({y-\chi_t^t(x)\over t^{1/\alpha}}\right)-{1\over t^{1/\alpha}}g^{(\lambda_t(x), \rho_t(x), \upsilon_t(x))}\left({y-\chi_t(x)\over t^{1/\alpha}}\right)\right)
\\&=:R_t^{2,1}(x,y)+R_t^{2,2}(x,y).
\ea
$$
We have
$$\ba
|R_t^{2,1}(x,y)|&=\left|\int_0^t\prt_s\left({1\over t^{1/\alpha}} g^{(\lambda_t(x), \rho_t(x), \overleftarrow\upsilon_t(x))}\left({{\kappa_{t-s}(y)-\chi_s^t(x)\over t^{1/\alpha}}}\right)\right)\, ds\right|
\\
&=\left|{1\over t^{1/\alpha}}\int_0^t \big(g^{(\lambda_t(x), \rho_t(x), \overleftarrow\upsilon_t(x))}\big)'\left({{\kappa_{t-s}(y)-\chi_s^t(x)\over t^{1/\alpha}}}\right){B_{t-s}(\kappa_{t-s}(y))-B_{t-s}(\chi_s^t(x))\over t^{1/\alpha}}\, ds\right|.
\ea
$$
Using \eqref{g_bound_x}, \eqref{Lip_B},  \eqref{flowA}, and \eqref{chi-t-chi-error3}, we get
$$\ba
&\left|{1\over t^{1/\alpha}}\big(g^{(\lambda_t(x), \rho_t(x), \overleftarrow\upsilon_t(x))}\big)'\left({{\kappa_{t-s}(y)-\chi_s^t(x)\over t^{1/\alpha}}}\right){B_{t-s}(\kappa_{t-s}(y))-B_{t-s}(\chi_s^t(x))\over t^{1/\alpha}}\right|
\\&\hspace{2cm}\leq {C\over t^{1/\alpha}}G^{(\alpha+1)}\left({\kappa_{t-s}(y)-\chi_s^t(x)\over t^{1/\alpha}}\right)\mathrm{Lip}\,(B_{t-s}){|\kappa_{t-s}(y)-\chi_s^t(x)|\over t^{1/\alpha}}
\\&\hspace{2cm}\leq {C(t-s)^{-1+\delta}\over t^{1/\alpha}}G^{(\alpha)}\left({y-\chi_t(x)\over t^{1/\alpha}}\right),
\ea
$$
which gives
$$
|R_t^{2,1}(x,y)|\leq Ct^{\delta}G_t^{(\alpha,\alpha,\alpha)}(\chi_t(x), y).
$$
Finally, we have
$$
g^{(\lambda_t(x), \rho_t(x), \overleftarrow\upsilon_t(x))}\left({y-\chi_t^t(x)\over t^{1/\alpha}}\right)=
g^{(\lambda_t(x), \rho_t(x), 0)}\left({y-\chi_t^t(x)-t^{1/\alpha}\overleftarrow\upsilon_t(x)\over t^{1/\alpha}}\right),
$$
$$
g^{(\lambda_t(x), \rho_t(x), \upsilon_t(x))}\left({y-\chi_t(x)\over t^{1/\alpha}}\right)=g^{(\lambda_t(x), \rho_t(x), 0)}\left({y-\chi_t(x)-t^{1/\alpha}\upsilon_t(x)\over t^{1/\alpha}}\right),
$$
and by \eqref{chi-t-chi}
$$\ba
 \chi_t^t(x)-\chi_t(x)&=t^{1/\alpha}\int_0^t \upsilon(\chi_r(x))\Big(W_\alpha(t;r)-W_\alpha(t;t-r)\Big)\, dr+Q_{t,t}(x)
 \\&=t^{1/\alpha}\upsilon_t(x)-t^{1/\alpha}\overleftarrow\upsilon_t(x)+Q_{t,t}(x).
 \ea
$$
Then by \eqref{chi-t-chi-error}
\be\label{chi-chi}
\left|\big(\chi_t^t(x)+t^{1/\alpha}\overleftarrow\upsilon_t(x)\big)-\big(\chi_t(x)+t^{1/\alpha}\upsilon_t(x)\big)\right|\leq Ct^{1/\alpha+\delta},
\ee
and similarly to the above estimates, using \eqref{g_bound_x},  \eqref{flowA}, and \eqref{chi-t-chi-error3}, we get
$$
|R_t^{2,2}(x,y)|\leq Ct^{\delta}G_t^{(\alpha,\alpha,\alpha)}(\chi_t(x), y).
$$
That is,
$$\ba
|p_t^0(x,y)-\wt p_t(x,y)|&\leq |R_t^{1}(x,y)|+|R_t^{2,1}(x,y)|+|R_t^{2,2}(x,y)|
\\&\leq Ct^{\delta_\zeta}G_t^{(\alpha,\alpha-\zeta,\alpha)}(\chi_t(x), y)+ Ct^{\delta}G_t^{(\alpha,\alpha,\alpha)}(\chi_t(x), y),
\ea
$$
which is just \eqref{R_tilde_bound}.
\end{proof}

\subsection{Properties of the kernels $G^{(\alpha,\beta,\gamma)}$.}\label{sA4}

\begin{proof}[Proofs of  \eqref{bint} and \eqref{G_kappa_bound}]
We have
$$
G_t^{(\alpha,\beta,\gamma)}(x,y)=G_t^{(\alpha,\beta,\beta)}(x,y)+t^{\beta/\alpha}\Big[|y-x|^{-\gamma-1}-|y-x|^{-\beta-1}\Big]1_{|y-x|>1},
$$
and
$$
G_t^{(\alpha,\beta,\beta)}(x,y)\leq t^{-1/\alpha}G^{(\beta)}\left(y-x\over t^{1/\alpha}\right).
$$
Since $G^{(\beta)}\in L_1(\Re)$, we get \eqref{bint}:
\be\label{bint2}
\int_{\Re}G_t^{(\alpha,\beta,\gamma)}(x,y)\, dx\leq C, \quad \int_{\Re}G_t^{(\alpha,\beta,\gamma)}(x,y)\, dy\leq C, \quad t\in (0, T].
\ee

Next, the kernel $G_t^{(\alpha,\beta,\gamma)}$ depends only on $(y-x)/t^{1/\alpha}$:
\be\label{FG}
G_t^{(\alpha,\beta,\gamma)}(x,y)=F_t^{(\alpha,\beta,\gamma)}\left({y-x\over t^{1/\alpha}}\right).
\ee
It is straightforward to verify that the corresponding function
\be\label{Fabg}
F^{(\alpha, \beta, \gamma)}_t(x)=    \left\{                                 \begin{array}{ll}{t^{-1/\alpha}},& |x|\leq (1\wedge t^{-1/\alpha}),\\
                                    {t^{-1/\alpha}} |x|^{-\beta-1}, & (1\wedge t^{-1/\alpha})<|x|\leq t^{-1/\alpha}),\\
                                 {t^{(\beta-\gamma-1)/\alpha}} |x|^{-\gamma-1}, & |x|>t^{-1/\alpha}
                                    \end{array}
                                  \right.
\ee
satisfies the analogues of \eqref{vague}, \eqref{G_mult}:
\be\label{vagueF}
F_t^{(\alpha,\beta,\gamma)}(x+v)\leq CF_t^{(\alpha,\beta,\gamma)}(x), \quad |v|\leq 1.
\ee
and for any $c>0$ there exists $C$ such that
\be\label{F_mult}
F_t^{(\alpha,\beta,\gamma)}(cx)\leq CF_t^{(\alpha,\beta,\gamma)}(x)
\ee
(the constants $C$ can be chosen the same for all $t\in (0, T]$). Using \eqref{flowCorr} with $s=t$ and \eqref{vagueF}, \eqref{F_mult},  we get
\be\label{G_kappa_chi}
G_t^{(\alpha,\beta,\gamma)}(x,\kappa_t(y))\leq CG_t^{(\alpha,\beta,\gamma)}(\chi_t(x),y).
\ee
Combined with \eqref{bint} this gives \eqref{G_kappa_bound}.
\end{proof}

We say that a non-negative kernel $H_t(x,y)$ has a \emph{sub-convolution property}, if for every $T>0$ there exists a constant $C$ such that
\begin{equation}\label{H0}
(H_{t-s}* H_s)(x,y)\leq C H_{t}(x,y), \quad t\in (0, T], \quad s\in (0, t), \quad x,y\in \Re.
\end{equation}

\begin{prop}\label{pAsub_conv} For arbitrary $\alpha, \beta, \gamma>0$, the kernel  $G_t^{(\alpha,\beta,\gamma)}(x,y)$
has a sub-convolution property.
\end{prop}
\begin{proof} We have   for $H_t(x,y)=G_t^{(\alpha,\beta,\gamma)}(x,y)$
$$
  \sup_{x,y}H_{t-s}(x,y)\leq \left({2\over t}\right)^{1/\alpha}
$$ for  $s<t/2$, and
$$
  \sup_{x,y}H_s(x,y)\leq \left({2\over t}\right)^{1/\alpha}
$$
otherwise. In both these cases we have by \eqref{bint2}
$$
\sup_{x,y}(H_{t-s}* H_s)(x,y)=\int_{\Re}H_{t-s}(x,z)H_s(z,y)\, dz\leq  Ct^{-1/\alpha}.
$$
 This proves \eqref{H0} for $x,y$ such that $|x-y|\leq 2t^{1/\alpha}$. Next,  $H_t(x,y)$ is positive and thus
$$\ba
(H_{t-s}*H_s)(x,y)&=\int_{\Re}H_{t-s}(x,z)H_s(z,y)\, dz
\\&\leq \int_{|x-z|>|x-y|/2}H_{t-s}(x,z)H_s(z,y)\, dz+\int_{|y-z|>|x-y|/2}H_{t-s}(x,z)H_s(z,y)\, dz.
\ea
$$
  The function $F_t^{(\alpha,\beta,\gamma)}(x)$ in the presentation \eqref{FG} of $H_t(x,y)=G_t^{(\alpha,\beta,\gamma)}(x,y)$, for  a fixed $t$,  depends only on $|x|$, and is a non-increasing function of $|x|$. Hence
$$
H_{t-s}(x,z)\leq H_{t-s}(x/2,y/2), \quad \hbox{ for } |x-z|>|x-y|/2\Leftrightarrow |x-z|>\left|{x\over 2}-{y\over 2}\right|.
$$
and
$$
H_{s}(z,y)\leq H_{s}(x/2,y/2), \quad \hbox{ for } |y-z|>|x-y|/2\Leftrightarrow |y-z|>\left|{x\over 2}-{y\over 2}\right|.
$$
Therefore by \eqref{bint2} 
$$\ba
(H_{t-s}* H_s)(x,y)&\leq \int_{|x-z|>|x-y|/2}H_{t-s}(x/2,y/2)H_s(z,y)\, dz+\int_{|y-z|>|x-y|/2}H_{t-s}(x,z)H_{s}(x/2,y/2)\, dz
\\&\leq C\Big(H_{t-s}(x/2,y/2)+H_{s}(x/2,y/2)\Big).
\ea
$$
Then for $|x-y|\geq 2t^{1/\alpha}$ we deduce
$$
(H_{t-s}* H_s)(x,y)\leq C((t-s)^{\beta/\alpha}+s^{\beta/\alpha})F^{(\beta,\gamma)}\left(y-x\over 2\right),
\quad
F^{(\beta,\gamma)}(x)=\left\{
                     \begin{array}{ll}
                       |x|^{-\beta-1}, & |x|\leq 1; \\
                       |x|^{-\gamma-1}, & |x|> 1.
                     \end{array}
                   \right.
$$
Clearly, $(t-s)^{\beta/\alpha}+s^{\beta/\alpha}\leq 2t^{\beta/\alpha}$ and
$$
F^{(\beta,\gamma)}(x/2)\leq C F^{(\beta,\gamma)}(x),
$$
which completes the proof of \eqref{H0} for $|x-y|\geq 2t^{1/\alpha}$.
\end{proof}

\begin{prop}\label{pAsub_conv_flow} The kernel
$$
H_t(x,y)=G_t^{(\alpha,\beta,\gamma)}(x,\kappa_t(y))
$$ has the sub-convolution property.
\end{prop}
\begin{proof} Using \eqref{G_kappa_chi} and  Proposition \ref{pAsub_conv}, we get
$$
(H_{s}* H_{t-s})(x,y)\leq C\int_{\Re}G_s^{(\alpha,\beta,\gamma)}(\chi_s(x),y')G_{t-s}^{(\alpha,\beta,\gamma)}(y',\kappa_{t-s}(y))\, dy'\leq C
G_t^{(\alpha,\beta,\gamma)}(\chi_s(x),\kappa_{t-s}(y)).
$$
Using  \eqref{flowCorr}, \eqref{vagueF}, and \eqref{F_mult}, we get similarly to \eqref{G_kappa_chi}
$$
G_t^{(\alpha,\beta,\gamma)}(\chi_s^s(x),\kappa_{t-s}(y))\leq CG_t^{(\alpha,\beta,\gamma)}(x,\kappa_{t}(y))=CH_t(x,y),
$$
which completes the proof.
\end{proof}


\begin{thebibliography}{1}

\bibitem{AJ07}  Y. A{\"{\i}}t-Sahalia and J. Jacod (2007), Volatility estimators for discretely sampled L\'{e}vy processes. \emph{Ann. Statist.} \textbf{35}, no. 1, 355--392.

\bibitem{BSW}   B. B\"ottcher, R. Schilling, and J. Wang (2013). L\'evy matters. III, volume 2099 of Lecture Notes in Mathematics.
Springer, Cham. \emph{L\'evy-type processes: construction, approximation and sample path properties}, With
a short biography of Paul L\'evy by Jean Jacod, L\'evy Matters.

\bibitem{Dit99} P. D. Ditlevsen (1999). Observation of $\alpha$-stable noise induced
millenial climate changes from an ice record. Geophysical Research Letters, \textbf{26}, no. 10, 1441--1444.


    \bibitem{EK86} S.\ N.\ Ethier, T.\ G.\ Kurtz (1986). \emph{Markov Processes: Characterization and Convergence. } Wiley, New York.
    \bibitem{Fr64} A.\ Friedman (1964). \emph{Partial differential equations of parabolic type.} Prentice-Hall,  New-York.
\bibitem{GKK} Iu. Ganychenko, V. Knopova, and    A.Kulik (2015).  Accuracy of discrete approximation for integral functionals of Markov processes,
\emph{Modern Stochastics: Theory and Applications,}
\textbf{2}, no. 4, 401--420.


\bibitem{G68} B.\ Grigelionis   (1968).  On a Markov property of Markov processes.   (Russian) {\it Liet. Matem. Rink}, {\bf  8(3)}, 489--502.

\bibitem{H15} L.\ Huang (2015). Density estimates for SDEs driven by tempered stable processes. arXiv:1504.04183

%\bibitem{AJ08} A\"{\i}t-Sahalia, Y. and Jacod, J. (2008), Fisher's information for discretely sampled L\'evy processes. Econometrica 76, 727--761.
%\bibitem{IKK16} Ivanenko,  Kohatsu-Higa, Kulik, LA(M)N property for diffusion models with low regularity coefficients (manuscript)
%\bibitem{Ja02} N.\ Jacob. \emph{Pseudo differential operators and Markov processes, II: Generators and their potential theory.} Imperial College Press, London, 2002.
\bibitem{KK13} V.\ Knopova and  A.\  Kulik, Intrinsic small time estimates for distribution densities of L\'evy processes.
 {\it Random Op. Stoch. Eq.} {\bf 21(4)}  (2013), 321--344.
\bibitem{KK15} V. \ Knopova and  A.\ Kulik  (2018).  Parametrix construction of the transition probability density of the solution to an SDE driven by $\alpha$-stable noise. {\itshape Ann. Inst. Henri Poincar\'e}.  {\bf 54(1)}, 100--140.

    \bibitem{KKS18} V.\ Knopova, A.\ Kulik, and R.\ Schilling. On the construction of a general stable-like Markov process  (maunuscript)


\bibitem{Ko89} A.\ N.\ Kochubei  (1989). Parabolic pseudo-differential equations, hypersingular integrals and Markov processes. \emph{Math. URSS Izestija.} \textbf{33}, 233--259.

\bibitem{KKK} A.\ N.\ Kochubei,     V.\ P.\ Knopova, and A.\ M.\ Kulik, \emph{Parametrix Methods for Equations with Fractional Laplacians}, \emph{to appear in}  Handbook Fractional Calculus with Applications, Springer.


\bibitem{Ko00}  V.\ N.\ Kolokoltsov  (2000). Symmetric Stable Laws and Stable-like Jump-Diffusions. \emph{Proc.
London Math. Soc.} \textbf{80}, 725--768.

\bibitem{Ko18}V.\ N.\ Kolokoltsov (2018).
\emph{Differential equations on measures and functional
spaces}, Birk\"auser.


\bibitem{Ko84a} T.\ Komatsu  (1984). On the martingale problem for generators of stable processes with perturbations.  Markov processes associated with certain integro-differential operators.  \emph{Osaka J. Math.} \textbf{21}, 113--132

\bibitem{Ko84b} T.\ Komatsu  (1984).    Pseudo-differential operators and Markov processes.   \emph{ J. Math. Soc. Japan} \textbf{36(3)}, 387--418.

\bibitem{KulStoRyz18} T. Kulczycki, M. Ryznar, and P. Sztonyk,
        Strong Feller property for SDEs driven by multiplicative cylindrical
        stable noise, arXiv:1811.05960.

 \bibitem{Ku18} A. Kulik (2019).  On weak uniqueness and distributional properties of a
solution to an SDE with $\alpha$-stable noise.  {\it Stoch. Proc. Appl.}, \textbf{129}, no. 2, 473--506.


\bibitem{KM16} A. Kulik, H. Masuda. Least Absolute Deviation estimator for a parameter in the drift term (maunuscript)



\bibitem{Kuh17-2}  F.\ K\"uhn (2017). {\it L\'evy-Type  Processes:  Moments,  Construction  and  Heat  Kernel  Estimates.}
  L\'evy Matters,  {\bf 5}, Springer.

\bibitem{R07} J. Rosinski (2007). Tempering stable processes, {\it Stoch. Proc. Appl.} 117,  677--707.


\bibitem{S18} K. Szczypkowski (2018). Fundamental solution for super-critical non-symmetric L\'evy-type operators, arXiv:1807.04257.


\bibitem{SCK} I.M. Sokolov, A.V. Chechkin, and J. Klafter  (2004). Fractional diffusion equation for a
power-law-truncated L\'evy process, \emph{Physica A} 336 245 -- 251.



\bibitem{TTW74} H.\ Tanaka, M.\ Tsuchiya, and S.\ Watanabe (1974). Perturbation of drift-type for L\'evy processes,  \emph{J. Math. Kyoto Univ.} \textbf{14}, 73-92.

\end{thebibliography}
   \end{document}